\pdfoutput=1
\documentclass[11pt]{article}
\usepackage{amssymb}
\usepackage{graphicx}
\usepackage{xcolor} 
\usepackage{tensor}
\usepackage{fullpage} 
\usepackage{amsmath}
\usepackage{amsthm}
\usepackage{verbatim}
\usepackage{mathrsfs}

\usepackage{scalerel,stackengine}
\stackMath
\newcommand\wwhat[1]{%
\savestack{\tmpbox}{\stretchto{%
  \scaleto{%
    \scalerel*[\widthof{\ensuremath{#1}}]{\kern-.6pt\bigwedge\kern-.6pt}%
    {\rule[-\textheight/2]{1ex}{\textheight}}
  }{\textheight}%
}{0.5ex}}%
\stackon[1pt]{#1}{\tmpbox}%
}
\usepackage{enumitem}
\setlist[enumerate]{leftmargin=1.5em}
\setlist[itemize]{leftmargin=1.5em}

\usepackage{seqsplit}

\definecolor{green}{rgb}{0,0.8,0} 

\newtheorem{maintheorem}{Theorem}

\newtheorem{theorem}{Theorem}[section]

\newtheorem{lemma}[theorem]{Lemma}
\newtheorem{proposition}[theorem]{Proposition}

\newtheorem{conjecture}[theorem]{Conjecture}

\theoremstyle{definition}
\newtheorem{definition}[theorem]{Definition}

\theoremstyle{remark}
\newtheorem{remark}[theorem]{Remark}

\numberwithin{equation}{section}
\newcommand{\RN}[1]{%
  \textup{\uppercase\expandafter{\romannumeral#1}}%
}

\newcommand{\lnrm}[1]
{\left\Vert#1\right\Vert}

\newcommand{\nnrm}[1]{{\vert\kern-0.25ex\vert\kern-0.25ex\vert #1 
    \vert\kern-0.25ex\vert\kern-0.25ex\vert}}

\newcommand{\supp}{{\mathrm{supp}}\,}

\newcommand{\ud}{\mathrm{d}}
\newcommand{\rd}{\partial}
\newcommand{\nb}{\nabla}

\newcommand{\alp}{\alpha}

\newcommand{\bfa}{{\bf a}}

\newcommand{\bfu}{{\bf u}}
\newcommand{\bfv}{{\bf v}}

\newcommand{\bfB}{{\bf B}}

\newcommand{\bfP}{{\bf P}}

\newcommand{\bbC}{\mathbb C}

\newcommand{\bbN}{\mathbb N}

\newcommand{\bbP}{\mathbb P}

\newcommand{\bbR}{\mathbb R}

\begin{document}

\title{Global well-posedness of the partially damped 2D MHD equations via a direct normal mode method for the anisotropic linear operator }
\author{Min Jun Jo, Junha Kim, Jihoon Lee} 
\date{\today}



\maketitle


\begin{abstract}
 We prove the global well-posedness of the 2D incompressible non-resistive MHD equations with a velocity damping term near the non-zero constant background magnetic field. To this end, we newly design a normal mode method of effectively leveraging the anisotropy of the linear propagator that encodes both the partially dissipative nature of the non-resistive MHD system and the stabilizing mechanism of the underlying magnetic field.  Isolating new key quantities and estimating them with themselves in an entangling way via the eigenvalue analysis based on Duhamel's formulation, we establish the global well-posedness for any initial data $(v_0,B_0)$ that is sufficiently small in a space rougher than $H^{4}\cap L^1$. This improves the recent work in SIAM J. Math. Anal. 47, 2630–2656 (2015) where the similar result was obtained provided that $(v_0,B_0)$ was small enough in a space strictly embedded in $H^{20}\cap W^{6,1}$.
\end{abstract}


\section{Introduction}

Plasma, the fourth state of matter after solid, liquid, and gas, accounts for the state of most visible matter in space - investigating the dynamics of plasma is crucial for our understanding of physics. The most famous example of the plasma entities is the Sun in our solar system. Coronal mass ejection from the Sun, which is known to be huge release of plasma, can trigger magnetic storms that would damage the communication satellites. There was an incident that a coronal mass ejection actually landed on the earth and stretched out the auroral zone, see \cite{Piel}. In astrophysics, plasma dynamics is an important subject of research.

The motion of plasmas can be effectively modeled \cite{Alf,Cab,Landau} by the incompressible MHD (magnetohydrodynamics) equations  
\begin{equation}  \label{mhd_origin}
\left\{
\begin{aligned} 
&\rd_t \bfv + \kappa (-\Delta)^{\alpha} \bfv + (\bfv\cdot \nb) \bfv - (\bfB \cdot \nabla)\bfB - \nabla p = 0, \\
&\rd_t\bfB + \mu (-\Delta)^{\alpha}\bfB + (\bfv \cdot \nabla)\bfB - (\bfB \cdot \nabla) \bfv = 0, \\
&\nabla \cdot \bfv = \nabla\cdot \bfB = 0,  
\end{aligned}
\right.
\end{equation}
in $\bbR^2$ with the divergence-free initial data $(\bfv_0,\bfB_0)$ for $\alp\geq 0.$ Here $\bfv$, $\bfB$, and $p$ denote the velocity vector field, the magnetic vector field, and the scalar pressure, respectively. The numbers $\kappa\geq 0$ and $\mu\geq 0$ are the diffusion coefficients, called the fluid viscosity and the ohmic resistivity. 

\subsection{Partial dissipation and inviscid damping for the MHD system}

It has been conjectured that energy of the system \eqref{mhd_origin} for $\alp=1$, equipped with a nontrivial background magnetic field, would be dissipated at a rate independent of the resistivity $\mu\geq 0$; in \cite{CC}, the conjecture was numerically backed up for the 2D case. This suggests that the stabilizing mechanism, which stems from the presence of the background magnetic field, be strong enough to allow us to ignore the effect of resistivity in view of energy dissipation. To see the heart of matter, we investigate the extreme case $\mu=0$ of \eqref{mhd_origin} which is the following non-resisitve MHD system
\begin{equation}  \label{mhd_origin2}
\left\{
\begin{aligned} 
&\rd_t \bfv + \kappa (-\Delta)^{\alpha} \bfv + (\bfv\cdot \nb) \bfv - (\bfB \cdot \nabla)\bfB - \nabla p = 0, \\
&\rd_t\bfB + (\bfv \cdot \nabla)\bfB - (\bfB \cdot \nabla) \bfv = 0, \\
&\nabla \cdot \bfv = \nabla\cdot \bfB = 0. 
\end{aligned}
\right.
\end{equation}
The above system is \emph{partially dissipative,} meaning that only certain types of motion are damped within the system. Here only the fluid velocity is damped by the diffusion term $(-\Delta)^{\alp}\mathbf{v}$. In contrast, when $\kappa>0$ and $\mu>0$, both the fluid velocity and the magnetic field are damped and so the original MHD system \eqref{mhd_origin} is \emph{fully dissipative.} While it is well-known that the system \eqref{mhd_origin} is globally well-posed \cite{ST} as long as we ensure its fully dissipative nature by setting up $\kappa>0$ and $\mu>0$, it remains open whether the classical solutions to the non-resistive (and so partially dissipative) system \eqref{mhd_origin2} develop finite time singularities or not. 

It turned out the difficulty due to the absence of the damping for $\mathbf{B}$ in \eqref{mhd_origin2} can be overcome by exploiting the stabilizing effect of the underlying constant magnetic field. Specifically, Lin, Xu, and Zhang (2015) in \cite{LZ2} proved the small global well-posedness of such system \eqref{mhd_origin2} with $\alp=1$ near the stationary state $(\bfv,\bfB)=(0,e_1)$. For the 3D case, see \cite{LZ}. After these breakthroughs, there have been many results regarding the effect of the underlying field $(0,e_1)$. One may refer to \cite{AZ,CL,CDW,CaoWu,DZ,JWY,Ren,Ting,Ting2} for the various related results. 

The stabilizing mechanism of the constant background magnetic field on $\mathbf{B}$ can be compared with \emph{inviscid damping} for the Euler equations near the Couette (or shear) flow in the sense that $\mathbf{B}$ equations themselves in \eqref{mhd_damp} are diffusion-less but the perturbation around $(0,e_1)$ mixes the phase as a whole, yielding the exponential time decay on the linearized level. See in \eqref{mhd_eq} how the extra linear structure of $\partial_1 B$ and $\partial_1 v$ is obtained in a system-entangling way by adopting the perturbative regime. One may also see \cite{CL} for the ideal MHD case.


 In this paper, we focus on the particular model with a damping velocity
\begin{equation}  \label{mhd_damp}
\left\{
\begin{aligned} 
&\rd_t \bfv + \bfv + (\bfv\cdot \nb) \bfv - (\bfB \cdot \nabla)\bfB - \nabla p = 0 \\
&\rd_t\bfB + (\bfv \cdot \nabla)\bfB - (\bfB \cdot \nabla) \bfv = 0 \\
&\nabla \cdot \bfv = \nabla\cdot \bfB = 0,  
\end{aligned}
\right.
\end{equation}
which is an end-case $(\alp,\mu)=(0,0)$  of the original MHD system \eqref{mhd_origin2}. One notices that the $L^2$ energy estimate for \eqref{mhd_damp} does not give the control over $\nb \bfv$ anymore unlike the $\alp=1$ case of \eqref{mhd_origin}. Despite such obstruction, near the stationary solution  $(\bfv,\bfB)=(0,e_1)$, not only the global small existence of the solutions to \eqref{mhd_damp} but also the temporal decay of such solutions were obtained in \cite{Wu} for the relatively higher-order Sobolev spaces. The goal of this paper is to improve the result of \cite{Wu} by establishing both the small global wellposedness and the time decay of the solutions in the lower-order Sobolev spaces. 



\subsection{Magnetic relaxation conjecture}

To put such energy dissipation of the partially dissipative MHD system into perspective, we introduce another phenomenon called \emph{magnetic relaxation}, which was originally conjectured for \eqref{mhd_origin2} by Arnol'd (1974) in \cite{Ar} and then carefully discussed by Moffatt (1985) in \cite{Moff}. The magnetic relaxation conjecture basically tells that the magnetic field $\bfB$ would asymptotically converge to some stationary Euler flow while the fluid particles will eventually stop due to the kinetic dissipation, and that during such convergence the core topology of $\bfB$ would be preserved as the one of the initial field $\bfB_0$.

As justification of our target model \eqref{mhd_damp} in view of the relaxation conjecture, we emphasize that even though originally the case $\alpha=1$ was the main target of \cite{Moff}, in the same paper, Moffatt still expected that various type of dissipation, including the velocity damping case $\alpha=0$, would do equally well in terms of the relaxation. As brifely noted in the previous section, the case $\alp=0$ does not allow us to have control over the velocity gradient anymore unlike the full Laplacian case $\alp=1$; so the extreme case $\alp=0$ might be more suitable to study magnetic relaxation, preventing any heavy reliance on the presence of the diffusion term. Such a view is well-aligned with our consideration of the specific model \eqref{mhd_damp}.

The heuristic argument suggested in \cite{Moff} for the relaxation problem simply follows from the fact that sufficiently regular solutions to \eqref{mhd_origin} with $\mu=0$ should satisfy the energy equality
 \begin{equation*}
     \frac{1}{2}\frac{\ud}{\ud t}(\|\bfv\|_{L^2}^2+\|\bfB\|_{L^2}^2)+\kappa\|(-\Delta)^{\frac{\alpha}{2}} \bfv\|_{L^2}^2 = 0
 \end{equation*}
 which means that the fluid viscosity dissipates the total energy of the solutions. But physically the zero resistivity, i.e. the perfect conductivity, allows the magnetic field to preserve the nontrivial topological structure of the initial data $\bfB_0$ over time, and thus, according to Moffatt in \cite{Moff}, the magnetic field energy should enjoy certain lower bound due to such inherent configuration of the magnetic field while the velocity $\bfv$ goes to $0.$ The magnetic field $\bfB$ will be finally frozen once the fluid particles stop, and that moment will be the time that $\bfB$ attains its minimum energy because if $\bfB$ is still unfrozen then Lorentz force make the fluid particles move again and so the corresponding kinetic energy stemming from the movement will be dissipated by the viscosity anyway until the fluid particles are ultimately immobilized. We can formalize such behavior as the following.
 
 \begin{conjecture}\label{conj}
 Any sufficiently regular solution $(\bfv,\bfB)$ to \eqref{mhd_origin2} exhibits the magnetic relaxation, i.e., $\bfB$ converges to some stationary Euler flow $\mathring{\bfB}$ as $t\to\infty$ while $|\bfv|_{L^2}$ decays to $0$ as $t\to\infty.$
 \end{conjecture}
 
 The difficulty of proving the conjecture arises in the fact that we do not even know whether the global existence result can be shown for the system \eqref{mhd_origin2} in both two and three dimensions. For the recent local well-posedness results with $\alp=1$, see \cite{CF}, \cite{Feff1}, and \cite{Feff2}.
 
 However, once we restrict ourselves to the vicinity of the simplest stationary solution $(\bfv,\bfB)=(0,e_1)$, we can establish not only the global existence for our main model \eqref{mhd_damp} but also the temporal decay of the corresponding solutions. See Theorem~\ref{thm1} and \eqref{rmk_decay}. Both results heavily depend on the following perturbation method. Looking at the unknowns as small perturbation around $(0,e_1),$ we write
 \begin{equation*}
     \bfv=0+v, \quad \bfB=e_1+B,
 \end{equation*}
 which transforms the original system \eqref{mhd_damp}  into our main target system
\begin{equation}  \label{mhd_eq}
\left\{
\begin{aligned} 
&\rd_tv + v + (v\cdot \nb) v - (B \cdot \nabla)B - \nabla p = \partial_1 B, \\
&\rd_tB + (v \cdot \nabla)B - (B \cdot \nabla) v = \partial_1v, \\
&\nabla \cdot v = \nabla \cdot B = 0. 
\end{aligned}
\right.
\end{equation} 
Note that the above equations have the additional \emph{linear} terms $\partial_1 B$ and $\partial_1 v$. Such extra linear structure of the equations allows us to prove the small global existence in our main theorem, and we can get even the temporal decay of the solutions \eqref{decay_rmk}. In other words, we witness that the solutions as small deviation from the equilibrium $(0,e_1)$ converge to the equilibrium asymptotically. This appears to be the magnetic relaxation phenomena of \cite{Ar,Moff} and simultaneously it is also the resistivity-independent energy dissipation conjectured in \cite{CC}. Then we reach the following question: Can one actually view such resistivity-independent energy dissipation of the solutions to \eqref{mhd_origin2} (including \eqref{mhd_damp}) as a part of the magnetic relaxation conjecture? Equivalently, for any background magnetic field that is a stationary solution to the system \eqref{mhd_origin} with zero velocity field, will the corresponding solutions to \eqref{mhd_origin2} be always dissipated and converge to the specified background field? The answer is not known, so far only the nonzero constant magnetic field was considered.

Very recently, Beekie, Friedlander, and Vicol in \cite{MRE} proved that the steady state $(0,e_1)$ is asymptotically stable in the so-called magnetic relaxation equations (MRE) with respect to the norm of $H^{m}$ for $m\geq 14$. The MRE was introduced by Moffatt in \cite{Moff,Moff2} to guarantee the energy dissipation in view of the relaxation phenomena of the non-resistive MHD equations. Note that our velocity damping case \eqref{mhd_damp} corresponds to the MRE with $\gamma=0$, which was the main target for the stability analysis in \cite{MRE}. See Chapter 5 in \cite{MRE}. This  corroborates our perspective on Conjecture~\ref{conj} as the generalized statement for the stabilizing mechanism induced by the presence of certain underlying nontrivial magnetic field. As a final remark, such a perturbative regime is consistent with the geometric view of 
\cite{CS} by Choffrut and Šverák where the richness of the \emph{nearby} steady states for the 2D Euler equations is shown. See also \cite{DE}.

\subsection{Previous work}

To the best of the authors' knowledge, currently \cite{Wu} is the only available result regarding the velocity damping case \eqref{mhd_damp}. In their paper \cite{Wu}, Wu, Wu, and Xu (2015) attained the small global unique existence for \eqref{mhd_eq} even with the temporal decay properties. To discuss their result more precisely, writing $B=\nb^{\perp}\psi$, one may define the norm for the class $Y_0$ for initial data by
\begin{equation*}
     \|(v_0,\psi_0)\|_{Y_0}:=\|\langle\nabla\rangle^{N} (v_0,\psi_0)\|_{L^2}+\|\langle\nabla\rangle^{6+}(v_0,\psi_0)\|_{L^1}+\|\langle\nabla\rangle^{6+}(v_1,\psi_1)\|_{L^1}
\end{equation*}
where $N\geq 20$ is a large number, $v_1$ and $\psi_1$ are defined in an involved way: they both contain several nonlinear terms that concerned some wave-type linear operators. Now we state the existence part of \cite{Wu}. For simplicity, we omit the statement for the properties of the \emph{low-order} derivatives. The below is a slightly rougher version of their original theorem.
\begin{theorem}[J. Wu, Y. Wu, X. Xu]
Let $N\geq 20$ and $\varepsilon<0.01$. There exists a sufficiently small $\delta>0$ such that the following holds. Suppose $\|(v_0,\psi_0)\|_{Y_0} \leq \delta$. Then there exists a unique global solution pair $(v,\psi)$ to \eqref{mhd_eq} with $B=\nabla^{\perp}\psi$ such that
        \begin{equation*}
            \sup_{t\geq1}t^{-\varepsilon}\|\langle\nabla\rangle^N (v,\nabla\psi)\|_{L^2} <\infty.
        \end{equation*}     
\end{theorem}
\noindent This was the first breakthrough for the case $\alp=0$ for \eqref{mhd_origin} in 2D. We point out that the requirement $N\geq 20$ appears to stem from the method they used, which will be presented shortly. The authors of \cite{Wu} utilized the wave-type linearized equations that were derived by taking the time derivative. This method originated in \cite{LZ} where Lagrangian coordinates were adopted to show anisotropic regularity propagation for the free transport equation that emerged in the stream function formulation of \eqref{mhd_origin2} with $\alp=1$. Rewriting \eqref{mhd_origin2} in the corresponding Lagrangian coordinates, certain damped wave equations of the form $$\partial_{tt}\Phi - \Delta \partial_t\Phi - \partial_1^2\Phi=0$$ were obtained as the linearized system of \eqref{mhd_origin2}. In \cite{Wu}, although Lagrangian formulation was not directly used unlike in \cite{LZ2}, the linear kernel estimates were done based on a similar wave-type system
$$\partial_{tt}\Phi +\partial_t \Phi -\partial_1^2\Phi=0$$
as the linearized system of \eqref{mhd_origin2}.

\subsection{Summary of main result}
We summarize the contributions of our main result as follows.
\begin{itemize}
    \item \textbf{Leveraging anisotropy.} The key difficulty lies  in the inherent anisotropy of the linear propagator. This hampers the direct analysis of the linearized equations of \eqref{mhd_eq} via a normal mode method; the corresponding eigenvectors are \emph{not} orthogonal and so one cannot recover the original representation of a vector from the inner products of the vector with the eigenvectors. By introducing the inverse matrix of the eigenvector matrix, we produce the anisotropic decompositions that are opted for bringing the temporal decay out of the specific linear propagator, see \eqref{decom_est_1} and \eqref{decom_est_2}.

    \item \textbf{Identification of the key quantities.} We perform a specific type of $H^m$ energy estimate that pinpoints the key quantities we need to control, which are $\|\partial_1 v\|_{L^\infty}$ and $\|B_2\|_{L^\infty}$. Those key quantities encode the anisotropic nature of the problem; a careful use of the incompressibility condition, combined with certain calculus inequalities, allows for clarifying the \emph{directional} information more suitably in view of partial dissipation.

    \item \textbf{Reduction of the Sobolev exponents.} In \cite{Wu}, the required Sobolev exponent $N\geq 20$ for the initial data was distant from the optimal Sobolev exponents that were naturally expected for the local well-posendess, cf. \cite{Feff1,Feff2}. See also Proposition \ref{loc_prop} in the next section. More precisely, the initial data was assumed to belong to a space strictly embedded in $H^{20}\cap W^{7,1}$; another notable condition is the $L^1$ assumption on the higher-order derivatives, which generally helps one leverage the linear kernel. In this work, we prove that it is sufficient to require the initial data to be in a space rougher than $H^{4}\cap L^1$.
    
    \item \textbf{Non-turbulent solution.} The highest $H^{20+}$ norm of the solution $(v,B)$ established in \cite{Wu} possibly grows in time with the order $t^{\varepsilon}.$ Such growth in time implies that there could be \emph{turbulence}, which means \emph{energy transfer} from low to high frequencies. Our main theorem guarantees that the highest Sobolev norm $H^m$ of $(v,B)$ does \emph{not} grow in time, preventing energy transfer from low frequencies to high frequencies in the perturbative regime near the background magnetic field.

\item \textbf{Methodological robustness.} Our method can be applied to various types of equations whose linear propagators feature certain symmetry. This is particularly true for the MHD system \eqref{mhd_origin} with any $0\leq\alp\leq1$ in the presence of the underlying constant magnetic field; one can immediately mimic the representation formula \eqref{df_u} and the decompositions \eqref{decom_est_1}-\eqref{decom_est_2}, correspondingly to the $\alp.$ The key is to reproduce the temporal decay estimates that are analogous to \eqref{Omg_1}-\eqref{Omg_3}. This can be done by dividing the domain according to the anisotropy of the given eigenvalues. See the proof of Lemma~\ref{lem_note}.
\end{itemize}

\subsection{Main result}
To state our main results in a more general function space, we start by introducing a spatially anisotropic norm for initial data. 
\begin{definition}
For $m>0$, we say $f=(f_1,f_2,f_3,f_4)\in X^m$ if
\begin{equation*}
    \|f\|_{X^m}:=\|f\|_{H^m}+\lnrm{\frac{\sqrt{|\xi_1|}}{|\xi|}\widehat{f}}_{L_{\xi_1}^2 L_{\xi_2}^1} + \lnrm{\frac{|(\widehat{f}_3,\widehat{f}_4)|}{|\xi|}}_{L^1} + \lnrm{\frac{|(\widehat{f}_3,\widehat{f}_4)|}{\sqrt{|\xi_1|}}}_{L^1} <\infty,
\end{equation*}
where $H^m$ denotes the standard Sobolev space of order $m$.
\end{definition}
\begin{remark}\label{X^m}
    $H^m\cap L^1$ is embedded in $X^m$. See Appendix~\ref{sec_rmk} for the proof.
\end{remark}

 Our target is the global well-posedness for \eqref{mhd_damp} near the nonzero background field $(0,e_1)$, which is equivalent to proving the small global existence and uniqueness for the perturbed system \eqref{mhd_eq}. Using the above definition of $X^m$, we state our main theorem as the following.

\begin{maintheorem}\label{thm1}
	Let $m \in \bbN$ with $m\geq4$ and $(v_0,B_0) \in X^m$ with $\mathrm{div}\,v_0 = \mathrm{div}\,B_0 = 0$. Then there exists a constant $\delta=\delta(m)>0$ such that the following holds. If 
	\begin{equation*}
	   \|(v_0,B_0)\|_{X^m} \leq \delta,
	\end{equation*}
	then there exists a unique global solution pair $(v,B)\in C([0,\infty);H^m(\bbR^2))$ to the system \eqref{mhd_eq}.
\end{maintheorem}

\begin{remark}
    The proof of Theorem \ref{thm1} is based on the three ingredients.
    \begin{enumerate}
        \item \emph{Isolate the key quantities:} We perform a specific kind of $H^m$ energy estimate. By cautiously using the divergence-free condition, $\partial_2$ is replaced by $\partial_1$, which allows us to pinpoint the key quantities that need to be controlled. Those key quantities appear on the right-hand side of \eqref{Em_ineq}; specifically, they are $\|\partial_1 v\|_{L^\infty}$ and $\|B_2\|_{L^\infty}$. The detailed process of the reduction to the key quantities can be found in the proof of Proposition \ref{prop_eng}.
        
        \item \emph{Exploit the anisotropy of the linear kernel:} We use the linear kernel property to estimate the key quantities; applying the Duhamel's formula, we leverage the form of the \emph{eigenvalues} and the eigenvectors to control the key quantities $\|\partial_1 v\|_{L^\infty}$ and $\|B_2\|_{L^\infty}.$ This is the main content of Section 3; See the statements and proofs of both Proposition~ \ref{prop_vderiv} and Proposition \ref{prop_B2}, for examples.
        
        \item \emph{Continuity argument:} We define a target quantity $G$ in \eqref{G}, which captures the growth-in-time of the Sobolev norms of the solution. Assuming the \emph{smallness} of the initial data, the estimates obtained in Section~\ref{sec_thm} guarantees that $G$ stays small for all time. Then the existence time $T_\ast>0$ for the unique solution established in Proposition \ref{loc_prop} can be extended to any arbitrarily large time. This finishes the proof.
    \end{enumerate}
    
\end{remark}





\begin{remark}\label{rmk_decay}
    (i) Assuming further that $(v_0,B_0)\in H^{5}\cap W^{2,1}$, one may show that
    \begin{equation}\label{decay_rmk}
        (1+t)^{\frac{3}{4}}\| v_1(t) \|_{L^2} + (1+t)^{\frac{5}{4}}\| v_2(t) \|_{L^2} + (1+t)^{\frac{1}{4}}\| B_1(t) \|_{L^2} + (1+t)^{\frac{3}{4}}\| B_2(t) \|_{L^2} \leq 1
    \end{equation}
    for all $t \geq 0$. It remains open whether the conditions, $H^4 \cap L^1$ for global well-posedness and $H^5 \cap W^{2,1}$ for the temporal decay, can be more relaxed or not.  \\
    (ii) The decay rates in \eqref{decay_rmk} are optimal. We briefly show after Proposition~\ref{prop_opt} that there exists a solution to the linearized equations of \eqref{mhd_eq} with the exact decay rates in \eqref{decay_rmk}.
\end{remark}

\subsection{Organization of paper}
The rest of this paper is organized as follows. In Section 2, we introduce and prove the key ingredients; the anisotropic vector decomposition, the corresponding anisotropic time decays, some interpolation inequalities, local well-posedness, and the key energy estimate. In section 3, we prove the existence of small global unique solutions. Appendix contains certain nontrivial technical observations.

\section{Preliminaries}
 We use the notation $A\lesssim B$ when $A\leq CB$ for some constant $C>0$ which is independent of the two real-valued quantities $A$ and $B$. We denote by $\|f\|_{L_T^p X}$ the mixed norm $\int_{0}^{T}\|f(t)\|_{X}^p\,\mathrm{d}t$ for any $p>0$ and for any spatial function class $X$. The Fourier transform of $f$ is expressed as $\widehat{f}$. For any complex number $z$, we use $\bar{z}$ as the notation for the conjugate of $z$.

We need the following interpolation inequality.
\begin{lemma}
	Let $d \in \bbN$, $p>0$, and $q>0$. Then there exists a constant $C>0$ such that
	\begin{align}
		\label{intp_1}\| f \|_{L^1(\bbR^d)} &\leq C \big\| |x|^{\frac d2+p} f(x) \big\|_{L^2(\bbR^d)}^{\frac q{p+q}} \big\| |x|^{\frac d2-q} f(x) \big\|_{L^2(\bbR^d)}^{\frac p{p+q}}
	\end{align}
	for all $f \in \mathscr{S}(\bbR^d)$. As a consequence, for any $g\in \mathscr{S}(\bbR^d)$, if $d=2$, we have
	\begin{equation}\label{intp_2}
	    \|\widehat{g}\|_{L^1(\bbR^2)} \leq C \|\partial_1 g\|_{H^1(\bbR^2)}^{\frac{1}{2}}\|g\|_{H^1(\bbR^2)}^{\frac{1}{2}}
	\end{equation}
	for some $C>0.$  
\end{lemma}
\begin{remark}
	The constant $C=C(d,p,q)$ tends to infinity as $p \to 0$ or $q \to 0$.
\end{remark}
\begin{proof}
	Since $f \in L^1(\bbR^d)$, there exists $r>0$ such that $\| f \|_{L^1(B(0;r))} = \| f \|_{L^1(\bbR^d \setminus B(0;r))}$. By H\"{o}lder's inequality we have
	\begin{equation*}
	    \| f \|_{L^1(\bbR^d \setminus B(0;r))} \leq \big\| |x|^{\frac d2+p} f(x) \big\|_{L^2(\bbR^d \setminus B(0;r))} \big\| |x|^{-\frac d2 -p} \big\|_{L^2(\bbR^d \setminus B(0;r))} \leq C(p)r^{-p}\big\| |x|^{\frac d2+p} f(x) \big\|_{L^2(\bbR^d)}.
	\end{equation*}
	Similarly, we can see
	\begin{equation*}
	    \| f \|_{L^1(B(0;r))} \leq \big\| |x|^{\frac d2-q} f(x) \big\|_{L^2(B(0;r))} \big\| |x|^{-\frac d2 + q} \big\|_{L^2(B(0;r))} \leq C(q)r^{q}\big\| |x|^{\frac d2-q} f(x) \big\|_{L^2(\bbR^d)}.
	\end{equation*}
	Combining the above estimates with $\| f \|_{L^1(\bbR^d)} = 2\| f \|_{L^1(\bbR^d \setminus B(0;r))}^{\frac {q}{p+q}} \| f \|_{L^1(B(0;r))}^{\frac {p}{p+q}}$, we obtain \eqref{intp_1}. This proves \eqref{intp_1}.
	
	It remains to show \eqref{intp_2}. We fix $\xi_2 \in \bbR$ and consider $\widehat{g}(\cdot,\xi_2) \in L^1(\bbR)$. By \eqref{intp_1}, we have $$\| \widehat{g}(\cdot,\xi_2) \|_{L^1(\bbR)} \leq C \| \partial_1 \widehat{g}(\cdot,\xi_2) \|_{L^2(\bbR)}^{\frac{1}{2}} \| \widehat{g}(\cdot,\xi_2) \|_{L^2(\bbR)}^{\frac{1}{2}}.$$ Applying H\"{o}lder's inequality to this, we obtain
	\begin{equation*}
	    \| \widehat{g} \|_{L^1(\bbR^2)} \leq C \| \| \partial_1 \widehat{g} \|_{L^2_{\xi_1}(\bbR)} \|_{L^1_{\xi_2}(\bbR)}^{\frac{1}{2}} \| \| \widehat{g} \|_{L^2_{\xi_1}(\bbR)} \|_{L^1_{\xi_2}(\bbR)}^{\frac{1}{2}}.
	\end{equation*}
	With the simple consequence of \eqref{intp_1} 
	\begin{equation*}
	    \| \widehat{f} \|_{L^1(\bbR)} \leq C \| f \|_{H^1(\bbR)},
	\end{equation*} \eqref{intp_2} follows. This finishes the proof.
\end{proof}

Viewing $v$ and $B$ not as two separate quantities but as one unknown $\bfu=(v,B)$, even without the presence of the magnetic field damping term, we still can observe the decay properties of the solutions to \eqref{mhd_eq} as a result of the underlying linear structure. To investigate such linearization effect more succinctly, we drop all the nonlinear terms in \eqref{mhd_eq} to obtain the corresponding linearized system
\begin{equation}\label{lin_system}
	\rd_t \widehat{\mathbf{u}} + J \widehat{\mathbf{u}} = 0, \qquad\qquad J := 
	\begin{pmatrix}
		1 & 0 & i\xi_1 & 0 \\
		0 & 1 & 0 & i\xi_1 \\
		i\xi_1 & 0 & 0 & 0 \\
		0 & i\xi_1 & 0 & 0
	\end{pmatrix}.
\end{equation}
From the simple computation
\begin{equation*}
	\det (\lambda I - J) = (\lambda^2 - \lambda + \xi_1^2)^2,
\end{equation*}
we get the four pairs of an eigenvalue and an eigenvector $\{(\lambda_{\pm}(\xi),\overline{\mathbf{a}_{\pm}^{j}(\xi)}) \}_{j=1,2}$ given by 
\begin{equation*}
	\lambda_{\pm} := \frac {1 \pm \sqrt{1 - 4 \xi_1^2}}2, \qquad \overline{\mathbf{a}_{\pm}^1} := \begin{pmatrix}
	i\xi_1 \\ 0 \\ -\lambda_{\mp} \\ 0
\end{pmatrix},
\qquad \overline{\mathbf{a}_{\pm}^2} := \begin{pmatrix}
	0 \\ i\xi_1 \\ 0 \\ -\lambda_{\mp}
\end{pmatrix}
\end{equation*}
for the linear operator $J$ such that there holds $J \overline{\mathbf{a}_{\pm}^j} = \lambda_{\pm} \overline{\mathbf{a}_{\pm}^j}$ for $j = 1,2$. Note that $\mathbf{a}^1_{+} = \mathbf{a}^1_{-}$ for $\xi_1 = \frac 12$, which makes problems. Since it holds
\begin{equation*}
	\frac 1{\xi_1 \sqrt{1-4\xi_1^2}} \begin{pmatrix}
		-i\lambda_{+} & i\lambda_{-} & 0 & 0 \\
		0 & 0 & -i\lambda_{+} & i\lambda_{-} \\
		\xi_1 & -\xi_1 & 0 & 0 \\
		0 & 0 & \xi_1 & -\xi_1
	\end{pmatrix}	\begin{pmatrix}
		i\xi_1 & 0 & -\lambda_{-} & 0 \\
		i\xi_1 & 0 & -\lambda_{+} & 0 \\
		0 & i\xi_1 & 0 & -\lambda_{-} \\
		0 & i\xi_1 & 0 & -\lambda_{+}
	\end{pmatrix} = I,
\end{equation*}
we have
\begin{equation*}
	\mathbf{u} = \sum_{\substack{\gamma=\pm\\ j=1,2}} \langle \mathbf{u},\mathbf{a}^j_{\gamma} \rangle \mathbf{b}^j_{\gamma}, \qquad \mathbf{b}_{\pm}^1 := \frac 1{\xi_1 \sqrt{1-4\xi_1^2}}\begin{pmatrix}
	\mp i\lambda_{\pm} \\ 0 \\ \pm \xi_1 \\ 0
\end{pmatrix},
\qquad \mathbf{b}_{\pm}^2 := \frac 1{\xi_1 \sqrt{1-4\xi_1^2}}\begin{pmatrix}
	0 \\ \mp i\lambda_{\pm} \\ 0 \\ \pm\xi_1
\end{pmatrix}.
\end{equation*}
Now, we recall the system \eqref{mhd_eq} and apply the above projection. Then we can write for $\bold{u} = (v,B)$ that $$\rd_t \langle \widehat{\mathbf{u}},\mathbf{a}_{\pm}^j \rangle \mathbf{b}_{\pm}^j + \lambda_{\pm} \langle \widehat{\mathbf{u}},\mathbf{a}_{\pm}^j \rangle \mathbf{b}_{\pm}^j + \langle \mathbf{P}(\xi) \wwhat{ (v\cdot \nabla)\mathbf{u}},\mathbf{a}_{\pm}^j \rangle \mathbf{b}_{\pm}^j + \langle \mathbf{P}(\xi) \wwhat{ (B\cdot \nabla)\widetilde{\mathbf{u}}},\mathbf{a}_{\pm}^j \rangle \mathbf{b}_{\pm}^j = 0, $$ where
\begin{equation*}
	\widetilde{\mathbf{u}} := (B_1,B_2,v_1,v_2)^\top, \qquad \mathbf{P}(\xi) := 
	\left(
		\begin{array}{c|c}
			P(\xi) & 0\\
			\hline
			0 & P(\xi)
		\end{array}
	\right).
\end{equation*}
Here, we use the notation $\langle \cdot,\cdot \rangle$ as the usual inner product in $\bbC^4$. According to Duhamel's principle, it follows
\begin{equation}\label{df_u}
	\begin{aligned}
		\langle \widehat{\mathbf{u}}(t),\mathbf{a}_{\pm}^j \rangle \mathbf{b}_{\pm}^j &= e^{-\lambda_{\pm}t} \langle \widehat{\mathbf{u}}_0,\mathbf{a}_{\pm}^j \rangle \mathbf{b}_{\pm}^j - \int_0^t e^{-\lambda_{\pm}(t-\tau)}\langle \mathbf{P}(\xi) \wwhat{ (v\cdot \nabla)\mathbf{u}}(\tau),\mathbf{a}_{\pm}^j \rangle \mathbf{b}_{\pm}^j \,\mathrm{d}\tau \\
		&\hphantom{\qquad\qquad}- \int_0^t e^{-\lambda_{\pm}(t-\tau)}\langle \mathbf{P}(\xi) \wwhat{ (B\cdot \nabla)\widetilde{\mathbf{u}}}(\tau),\mathbf{a}_{\pm}^j \rangle \mathbf{b}_{\pm}^j \,\mathrm{d}\tau.
	\end{aligned}
\end{equation}
A significant problem of this formula is the unboundedness of $|\bold{b}_{\pm}^j|$. To overcome it, we employ a simple formulas
\begin{equation}\label{decom_est_1}
		\sum_{\gamma=\pm}e^{-\lambda_{\gamma}t} \langle  \widehat{\mathbf{f}},\mathbf{a}_{\gamma}^2 \rangle \langle \mathbf{b}_{\gamma}^2, e_2 \rangle = (e^{-\lambda_{-}t} - e^{-\lambda_{+}t}) \langle  \widehat{\mathbf{f}},\mathbf{a}_{-}^2 \rangle \langle \mathbf{b}_{-}^2, e_2 \rangle + e^{-\lambda_{+}t} \langle \widehat{\mathbf{f}}, e_2 \rangle
\end{equation} and
\begin{equation}\label{decom_est_2}
		\sum_{\gamma=\pm} e^{-\lambda_{\gamma}t} \langle  \widehat{\mathbf{f}},\mathbf{a}_{\gamma}^2 \rangle \langle \mathbf{b}_{\gamma}^2, e_4 \rangle = (e^{-\lambda_{-}t} - e^{-\lambda_{+}t}) \langle  \widehat{\mathbf{f}},\mathbf{a}_{-}^2 \rangle \langle \mathbf{b}_{-}^2, e_4 \rangle + e^{-\lambda_{+}t} \langle \widehat{\mathbf{f}}, e_4 \rangle,
\end{equation}
where $\widehat{\bold{f}} \in \bbC^4$. The followings are some useful estimates for this approach. By the definitions of $\lambda_{\pm}$, $\bold{a}_{\pm}^j$, and $\bold{b}_{\pm}^j$, we have
\begin{equation*}
	|e^{-\lambda_{+}t}| \leq e^{-\frac t2}, \qquad |e^{-\lambda_{-}t}| \leq \begin{cases} e^{-\frac t2}, \qquad |\xi_1| \geq \frac 12, \\ e^{-\xi_1^2t}, \qquad |\xi_1| \leq \frac 12, \end{cases}
\end{equation*}
\begin{equation*}
	|\bold{a}_{\pm}^1|^2 = |\bold{a}_{\pm}^2|^2 = \xi_1^2 + |\lambda_{\mp}|^2, \qquad |\langle \bold{b}_{\pm}^2,e_2 \rangle|^2 = \frac 1{|1-4\xi_1^2|}, \qquad |\langle \bold{b}_{\pm}^2,e_4 \rangle|^2 = \frac {|\lambda_{-}|^2}{\xi_1^2|1-4\xi_1^2|}.
\end{equation*}
To cancel out the singularity at $\xi_1 = \frac 12$, let us define strip domains
	\begin{equation}\label{domain}
		\Omega_1 := \{ \xi \in \bbR^2 ; |\xi_1| \geq \frac 12 \}, \qquad \Omega_2 := \{ \xi \in \bbR^2 ; \frac 12 \geq |\xi_1| \geq \frac 14 \}, \qquad \Omega_3 := \{ \xi \in \bbR^2 ; |\xi_1| \leq \frac 14 \}
	\end{equation}
	so that it holds $\bbR^2 = \Omega_1 \cup \Omega_2 \cup \Omega_3$. Then, the following lemma holds.
\begin{lemma}\label{lem_note}
	Let $f=(f_1,f_2,f_3,f_4) \in \bbC^4$. Then, there exists a constant $C>0$ such that
	\begin{equation}\label{Omg_1}
		|(e^{-\lambda_{-}t} - e^{-\lambda_{+}t}) \langle f,\mathbf{a}_{-}^2 \rangle \langle \mathbf{b}_{-}^2, e_2 \rangle| + |(e^{-\lambda_{-}t} - e^{-\lambda_{+}t}) \langle f,\mathbf{a}_{-}^2 \rangle \langle \mathbf{b}_{-}^2, e_4 \rangle| \leq C e^{-\frac t4} |f|, \qquad \xi \in \Omega_1,
	\end{equation}
	\begin{equation}\label{Omg_2}
		|(e^{-\lambda_{-}t} - e^{-\lambda_{+}t}) \langle  f,\mathbf{a}_{-}^2 \rangle \langle \mathbf{b}_{-}^2, e_2 \rangle| + |(e^{-\lambda_{-}t} - e^{-\lambda_{+}t}) \langle  f,\mathbf{a}_{-}^2 \rangle \langle \mathbf{b}_{-}^2, e_4 \rangle| \leq Ce^{-\frac t4} |f|, \qquad \xi \in \Omega_2,
	\end{equation}
	\begin{equation}\label{Omg_4}
		|(e^{-\lambda_{-}t} - e^{-\lambda_{+}t}) \langle  f,\mathbf{a}_{-}^2 \rangle \langle \mathbf{b}_{-}^2, e_2 \rangle| \leq Ce^{-\xi_1^2t}(|\xi_1^2f_2| + |\xi_1f_4|), \qquad \xi \in \Omega_3,
	\end{equation}
	\begin{equation}\label{Omg_3}
		|(e^{-\lambda_{-}t} - e^{-\lambda_{+}t}) \langle  f,\mathbf{a}_{-}^2 \rangle \langle \mathbf{b}_{-}^2, e_4 \rangle| \leq Ce^{-\xi_1^2t}(|\xi_1f_2| + |f_4|), \qquad \xi \in \Omega_3.
	\end{equation}
\end{lemma}
\begin{proof}
	To show \eqref{Omg_1}, we divide the domain $\Omega_1$ into $\{ \xi \in \bbR^2 ; |\xi_1| \geq \frac 34 \}$ and $\{ \xi \in \bbR^2 ; \frac 34 \geq |\xi_1| \geq \frac 12 \}$. On the first set, we can see $ |e^{-\lambda_{-}t} - e^{-\lambda_{+}t}| \leq Ce^{-\frac{t}{2}}$, $|\bfa_{-}^2|^2 \leq C\xi_1^2$, and $$|\langle \bold{b}_{\pm}^2,e_2 \rangle|^2 + |\langle \bold{b}_{\pm}^2,e_4 \rangle|^2 = \frac {1+\xi_1^2}{|1-4\xi_1^2|} \leq C.$$ This implies \eqref{Omg_1}. On the second set, we note by the mean value theorem that $$\frac{|e^{-\lambda_{-}t} - e^{-\lambda_{+}t}|}{|\sqrt{1-4\xi_1^2}|} = \frac{|e^{-\lambda_{-}t} - e^{-\lambda_{+}t}|}{|\lambda_{+} - \lambda_{-}|} \leq t e^{-\frac{t}{2}} \leq C e^{-\frac{t}{4}}.$$ With the fact $|\xi_1| \sim 1$, we can deduce \eqref{Omg_1}. Since \eqref{Omg_2} can be obtained similarly, we omit the details. It remains to show \eqref{Omg_4} and \eqref{Omg_3}. We note $|\langle  \widehat{\mathbf{f}},\mathbf{a}_{-}^2 \rangle| \leq |i\xi_1 f_2| + |\lambda_{+} f_4| \leq |\xi_1 f_2| + C|f_4|$ and $$|\langle \bold{b}_{\pm}^2,e_2 \rangle|^2 \leq \frac {\xi_1^2}{|1-4\xi_1^2|}, \qquad |\langle \bold{b}_{\pm}^2,e_4 \rangle|^2 \leq \frac {1}{|1-4\xi_1^2|}$$ for $\xi \in \Omega_3$. 
 The above two inequalities correspond to \eqref{Omg_4} and \eqref{Omg_3}, respectively. Then with the estimate
	\begin{equation}\label{omg_est}
        \lambda_{-} = \frac{2\xi_1^2}{1+\sqrt{1-4\xi_1^2}} \geq \xi_1^2, \qquad \xi \in \Omega_2 \cup \Omega_3,
    \end{equation}
	we can deduce  both\eqref{Omg_4} and \eqref{Omg_3}. This completes the proof.
\end{proof}

\begin{lemma}\label{lem_decay2}
	Let $j \in \bbN \cup \{0\}$. Then there exists a constant $C>0$ such that
	\begin{equation}\label{decay_est2}
		\big\| \xi_1^j e^{-\lambda_{\pm}t} \widehat{f} \big\|_{L^2} \leq C (1+t)^{-\frac{j}{2}-\frac{1}{4}} \| f \|_{H^{j} \cap L^1_{x_1}L^2_{x_2}}
	\end{equation}
	for all $t\geq0$ and $f \in H^{j} \cap L^1_{x_1}L^2_{x_2}(\bbR^2)$.
\end{lemma}
\begin{proof}
    For $t \leq 1$, using $|e^{-\lambda_{\pm}(\xi)t}| \leq 1$, we have
    \begin{equation*}
		\big\| |\xi_1|^j |e^{-\lambda_{\pm}t}| |\widehat{f}| \big\|_{L^2} \leq \| \partial_1^j f \|_{L^2} \leq C (1+t)^{-\frac{j}{2}-\frac{1}{4}}\| \partial_1^j f \|_{L^2}.
    \end{equation*}
    
\noindent For $t \geq 1$, since $|e^{-\lambda_{\pm}(\xi)t}| = e^{-\frac {t}{2}}$ holds when $\xi \in \Omega_1$, we get
    \begin{equation*}
		\big\| |\xi_1|^j |e^{-\lambda_{\pm}(\xi)t}| |\widehat{f}(\xi)| \big\|_{L^2(\bbR^2 \setminus \Omega)} \leq e^{-\frac{t}{2}} \| \partial_1^j f \|_{L^2} \leq C (1+t)^{-\frac{j}{2}-\frac{1}{4}} \| \partial_1^j f \|_{L^2}.
    \end{equation*}
    Similarly, by $|e^{-\lambda_{+}(\xi)t}| = e^{-\frac {t}{2}}$ for $\xi \in \Omega_2 \cup \Omega_3$, we have
    \begin{equation*}
		\big\| |\xi_1|^j |e^{-\lambda_{+}(\xi)t}| |\widehat{f}(\xi)| \big\|_{L^2(\Omega)} \leq C (1+t)^{-\frac{j}{2}-\frac{1}{4}} \big\| \partial_1^j f \big\|_{L^2}.
    \end{equation*}
    Meanwhile, from \eqref{omg_est} and
    \begin{equation*}
        \big\| e^{-\xi_1^2 t} \widehat{f}(\cdot,\xi_2) \big\|_{L^2_{\xi_1}} \leq C (1+t)^{-\frac{1}{4}} \big\| \widehat{f}(\cdot,\xi_2) \big\|_{L^{\infty}_{\xi_1}},
    \end{equation*}
    we have
    \begin{equation}\label{L2_decay_est}
    \begin{aligned}
    	\big\| |\xi_1|^j |e^{-\lambda_{-}(\xi)t}| |\widehat{f}(\xi)| \big\|_{L^2(\Omega)} &\leq C (1+t)^{-\frac{j}{2}} \big\| e^{-\xi_1^2 \frac t2} \widehat{f} \big\|_{L^2} \\
    	&\leq C (1+t)^{-\frac{j}{2}-\frac {1}{4}} \| \widehat{f} \|_{L^{\infty}_{\xi_1} L^2_{\xi_2}} \\
    	&\leq C (1+t)^{-\frac{j}{2}-\frac {1}{4}} \| f \|_{L^1_{x_1}L^2_{x_2}}.    
    \end{aligned}
	\end{equation}
	Note by the Sobolev embedding theorem that
	\begin{equation*}
		\| f \|_{L^{1}_{x_1} L^2_{x_2}} \leq C \| f \|_{W^{s,1}}, \qquad s>\frac{1}{2}.
    \end{equation*}
    Combining the above estimates, we obtain \eqref{decay_est2}. This completes the proof.
\end{proof}
Next, we show that \eqref{decay_est2} is optimal by providing the following proposition. Indeed, the $L^2$-norm decay rates of $(v,B)$ in \eqref{decay_rmk} are exactly the same with the decay rates of the decay rates of the solutions to the linearized system of \eqref{mhd_eq} (see Proposition~\ref{prop_opt2}).
\begin{proposition}\label{prop_opt}
	There exist $f^* \in \mathscr{S}(\bbR^2)$ and $C^*>0$ such that
	\begin{equation*}
		\big\| \xi_1^j e^{-\lambda_{-}t} \widehat{f}^* \big\|_{L^2} \geq C^* (1+t)^{-\frac{j}{2}-\frac{1}{4}}
	\end{equation*}
	for all $j \in \bbN \setminus \{0\}$.
	\end{proposition}
\begin{proof}
    We consider a cut-off function $\sigma \in C^{\infty}_c([0,1/4))$ with $\sigma(0) = 1$. For any complex-valued functions $\varphi \in L^2(\bbR)$ and $\phi \in C^{\infty}(\bbR)$ with $|\phi(0)| > 0$, we let $f^* \in \mathscr{S}(\bbR^2)$ satisfying $\widehat{f}^* = \phi(\xi_1)\sigma(|\xi_1|) \varphi(\xi_2)$. It suffices to show that
	\begin{equation*}
	\big\| \xi_1^j e^{-\lambda_{-}t} \widehat{f}^* \big\|_{L^2} \geq C^* t^{-\frac{j}{2}-\frac{1}{4}}, \qquad t \geq C.
	\end{equation*}
    Since $\supp(\widehat{f}^*) \subset \Omega$ holds, we have
    \begin{equation*}
	\big\| \xi_1^j e^{-\lambda_{-}t} \widehat{f}^* \big\|_{L^2} \geq \| \xi_1^j e^{-2\xi_1^2 t} \phi(\xi_1)\sigma(|\xi_1|) \|_{L^2_{\xi_1}} \| \varphi(\xi_2) \|_{L^2_{\xi_2}}.
	\end{equation*}
	Let $\eta := \xi_1 \sqrt{t}$. We note that $|\phi(\eta/\sqrt{t}) \sigma(\eta/\sqrt{t})| \geq \frac 12 |\phi(0)\sigma(0)|>0$ for all $|\eta| \leq 1$ and $t \geq C$. Thus, we have
    \begin{equation*}
	\| \xi_1^j e^{-2\xi_1^2 t} \phi(\xi_1) \sigma(|\xi_1|) \|_{L^2_{\xi_1}} = t^{-\frac j2-\frac 14} \| \eta^j e^{-2\eta^2} \phi(\eta/\sqrt{t}) \sigma(\eta/\sqrt{t}) \|_{L^2_{\eta}} \geq  C^* |\phi(0)\sigma(0)| t^{-\frac j2-\frac 14}
	\end{equation*}
	and obtain the claim. This completes the proof.
\end{proof}

It remains to prove that such temporal decay can be established for the actual solutions to the linearized equations.  

\begin{proposition}\label{prop_opt2}
    There exists an initial data $\bfu_0=(v_0,B_0) \in \mathscr{S}(\bbR^2)$ with $\operatorname{div}v_0=\operatorname{div}B_0=0$ such that the solution $\bfu=(v,B)$ to the linearized system \eqref{lin_system} has the exact $L^2$-norm decay rate in \eqref{decay_rmk}.
\end{proposition}
\begin{proof}
     We consider the initial data 
    \begin{equation*}
        \widehat{\mathbf{u}}_0 = (\widehat{v}_0,\widehat{B}_0) = \sigma(|\xi_1|) \sigma(|\xi_2|) (\xi_2e_1 - \xi_1 e_2 + \xi_2e_3 - \xi_1e_4),
    \end{equation*} where $\sigma \in C^{\infty}_c([0,1/4))$ is a cut-off function with $\sigma(0)=1$. Then,  the exact solution to the linearized system of \eqref{mhd_eq} can be written as $(\widehat{v},\widehat{B}) = \sum_{j=1,2} e^{-\lambda_{\pm}t} \langle \widehat{\mathbf{u}}_0,\mathbf{a}_{\pm}^j \rangle \mathbf{b}_{\pm}^j$. We only prove here that
    \begin{equation*}
        \| v_1(t) \|_{L^2} \geq \tilde{C}t^{-\frac{3}{4}}, \qquad t \geq C,
    \end{equation*}
    because we can similarly show
    \begin{equation*}
        \| v_2(t) \|_{L^2} \geq \tilde{C}t^{-\frac{5}{4}}, \qquad \| B_1(t) \|_{L^2} \geq \tilde{C}t^{-\frac{1}{4}}, \qquad \| B_2(t) \|_{L^2} \geq \tilde{C}t^{-\frac{3}{4}}, \qquad t \geq C.
    \end{equation*}
    Since $\widehat{v}_1 = e^{-\lambda_{\pm}t} \langle \widehat{\mathbf{u}}_0,\mathbf{a}_{\pm}^1 \rangle \langle \mathbf{b}_{\pm}^1,e_1 \rangle$ and $\supp \mathbf{u}_0 \subset \Omega$, it follows
    \begin{align*}
        \| v_1(t) \|_{L^2} &\geq \| e^{-\lambda_{-}t} \langle \widehat{\mathbf{u}}_0,\mathbf{a}_{-}^1 \rangle \langle \mathbf{b}_{-}^1,e_1 \rangle \|_{L^2} - \| e^{-\lambda_{+}t} \langle \widehat{\mathbf{u}}_0,\mathbf{b}_{+}^1 \rangle \langle \mathbf{b}_{+}^1,e_1 \rangle \|_{L^2} \\
        &\geq \| e^{-\lambda_{-}t} \langle \widehat{\mathbf{u}}_0,\mathbf{b}_{-}^1 \rangle \langle \mathbf{b}_{-}^1,e_1 \rangle \|_{L^2} - Ce^{-\frac t2}.
    \end{align*}
    Note that
    \begin{equation*}
       \langle \widehat{\mathbf{u}}_0,\mathbf{b}_{-}^1 \rangle \langle \mathbf{b}_{-}^1,e_1 \rangle = \sigma(|\xi_1|) \sigma(|\xi_2|) \xi_2 (\frac{i\xi_1}{\sqrt{\lambda_{+}}} - \sqrt{\lambda_{+}}) \frac{i\xi_1}{\sqrt{\lambda_{+}}} = \xi_1 (-\frac{\xi_1}{\lambda_{+}} - i) \sigma(|\xi_1|) \sigma(|\xi_2|) \xi_2.
    \end{equation*}
    Applying Proposition~\ref{prop_opt}, we can see
    \begin{align*}
        \| e^{-\lambda_{-}t} \langle \widehat{\mathbf{u}}_0,\mathbf{b}_{-}^1 \rangle \langle \mathbf{b}_{-}^1,e_1 \rangle \|_{L^2} \geq \tilde{C}t^{-\frac{3}{4}}, \qquad t \geq C.
    \end{align*}
    This completes the proof.
\end{proof} 

\subsection{Local well-posedness}

\begin{proposition}\label{loc_prop}For $s>2,$ and the initial data $(v_0,B_0)\in H^{s}(\bbR^2)$ with $\nb \cdot v_0 = \nb \cdot B_0 = 0$, there exists a maximal time of existence $T_{\ast}=T_{\ast}(s,\|v_0\|_{H^s},\|B_0\|_{H^s})>0$ such that the equations \eqref{mhd_eq} have a unique solution $(v,B)\in C([0,T_{\ast});H^s(\bbR^2)).$
\end{proposition}

\begin{proof}
    For brevity, here we show only the simple \emph{a priori} $H^s$ energy inequality. Following the proof in \cite{Feff1}, one can modify this to prove the proposition with Fourier truncation or standard mollification.
    
    Assume that $(v,B)$ is the sufficiently regular solution to \eqref{mhd_eq}. Define $\Lambda^{s}$ by the Fourier symbol $|\xi|^{s}$ for $s>2$. Apply $\Lambda^s$ to the both equations for $v$ and $B$ in \eqref{mhd_eq}. Then take the $L^2$ inner product of the first equation for $v$ with $\Lambda^s v$ and the inner product of the second equation for $B$ with $\Lambda^s B$. Adding them together yields, combined with the cancellation property
    $$\int \Lambda^s (\partial_1 B) \Lambda^s v + \int \Lambda^s (\partial_1 v) \Lambda^s B =0  $$
    stemming from integration by parts for $\partial_1$,
    \begin{equation}\label{Hs}
        \frac{1}{2}\frac{\ud}{\ud t}(\|\Lambda^s v\|_{L^2}^2+\|\Lambda^s B \|_{L^2}^2) + \|\Lambda^s v\|_{L^2}^2 = N_1 + N_2 + N_3 + N_4
    \end{equation}
where we have
\begin{equation*}
    \begin{split}
        N_1:&=\int \Lambda^s [B\cdot \nb B] \Lambda^s v \leq c\|\nb B\|_{L^\infty}\|B\|_{H^s}\|v\|_{H^s} \\
        N_2:&= -\int \Lambda^s[v\cdot\nb v] \Lambda^s v \leq c\|\nb v\|_{L^\infty}\|v\|_{H^s}^2 \\
        N_3:&= \int \Lambda^s[B\cdot \nb v] \Lambda^s B \leq c(\|\nb B\|_{L^\infty}\|v\|_{H^s} + \|\nb v\|_{L^\infty}\|B\|_{H^s})\|B\|_{H^s} \\
        N_4:&= \int -\Lambda^s[v\cdot\nb B] \Lambda^s B \leq c(\|\nb v\|_{L^\infty}\|B\|_{H^s}+\|\nb B\|_{L^\infty}\|v\|_{H^s})\|B\|_{H^s}
    \end{split}
\end{equation*}
by the usual Kato-Ponce type commutator estimates, see \cite{KP}. We add \eqref{Hs} to the $L^2$ energy equality
 \begin{equation}
     \frac{1}{2}\frac{\ud}{\ud t}(\|v\|_{L^2}^2+\|B\|_{L^2}^2)+ \|v\|_{L^2}^2 = 0
 \end{equation}
 and then apply the Sobolev inequalities $\|\nb v\|_{L^\infty} \leq C(s) \|v\|_{H^s}$ and $\|\nb B\|_{L^\infty} \leq C(s) \|B\|_{H^s}$ for $s>2$ in two dimensions to obtain
 \begin{equation*}
     \frac{1}{2}\frac{\ud}{\ud t}(\|v\|_{H^s}^2+\|B \|_{H^s}^2) + \|v\|_{H^s}^2 \leq c (\|B\|_{H^s}^2\|v\|_{H^s} + \|v\|_{H^s}^3).
 \end{equation*}
 Young's product inequality finally yields
 \begin{equation*}
     \frac{\ud}{\ud t}(\|v\|_{H^s}^2+\|B \|_{H^s}^2) + \|v\|_{H^s}^2 \leq c (\|v\|_{H^s}^2+\|B\|_{H^s}^2)^2.
 \end{equation*}
 Setting up $Y(t)=\|v\|_{H^s}^2+\|B\|_{H^s}^2$, we can view it as
 \begin{equation*}
     \frac{\ud}{\ud t}Y(t) \leq c Y(t)^2
 \end{equation*}
 which is the desired closed form that would lead to the application of Gronwall's inequality. \end{proof}

\subsection{Energy estimates}
Let $(v,B)$ and $T^*$ be the classical solution and the maximal time of existence in Proposition~\ref{loc_prop}. For simplicity, we use the notation
\begin{equation}\label{E_def}
	E(t)^2 := \| v(t) \|_{H^m}^2 + \| B(t) \|_{H^m}^2,
\end{equation}
\begin{equation}\label{A_def}
	A(t) := \sum_{0 \leq |\alpha| \leq m-1} \int_{\bbR^2} \partial^\alpha B(t,x) \cdot \partial_1 \partial^\alpha v(t,x) \,\mathrm{d}x,
\end{equation}
for $m \in \bbN$. We can easily verify that
\begin{equation}\label{A_est}
	|A(t)| \leq \frac 12 E(t)^2.
\end{equation}
\begin{proposition}\label{prop_eng}
	Let $m \in \bbN$ with $m>2$. Suppose that $(v,B) \in C([0,\infty);H^m(\bbR^2))$ is a classical solution of \eqref{mhd_eq}. Then there exists a constant $C=C(m)>0$ such that
	\begin{equation}\label{Em_ineq}
		\begin{gathered}
			\frac {\mathrm{d}}{\mathrm{d}t} \left( E(t)^2 + A(t) \right) + \frac 12 \| v(t) \|_{H^m}^2 + \frac 12 \| \partial_1 B(t) \|_{H^{m-1}}^2 \leq C E(t) \left( \| v(t) \|_{H^m}^2 + \| \partial_1 B(t) \|_{H^{m-1}}^2 \right)\\
			+ C \| \partial_1v(t) \|_{L^\infty} \| B(t) \|_{H^m}^2 + C \| B_2(t) \|_{L^\infty} \| \partial_1 B(t) \|_{H^{m-1}} \| B(t) \|_{H^m}
		\end{gathered}
	\end{equation}
	for all $t \geq 0$.
\end{proposition}
\begin{proof}
	We multiply \eqref{mhd_eq} by $(v,B)$ and integrate it over $\bbR^2$.  From $\mathrm{div}\, v = \mathrm{div}\, B = 0$, we obtain
	\begin{gather*}
		\frac 12 \frac {\mathrm{d}}{\mathrm{d}t} \left( \|v\|_{L^2}^2 + \|B\|_{L^2}^2 \right) + \|v\|_{L^2}^2 =0
	\end{gather*}
 due to integration by parts. For any multi-index $\alp$ with $1 \leq |\alpha| \leq m$, multiplying \eqref{mhd_eq} by $\partial^{2\alpha}(v,B)$ and integrating it over $\bbR^2$, we further have
	\begin{equation} \label{energy_alpha}
	\begin{split}
		\frac 12 \frac {\mathrm{d}}{\mathrm{d}t} \left( \|\partial^\alp v\|_{L^2}^2 +\|\partial^\alp B\|_{L^2}^2 \right) + \|\partial^\alpha v\|_{L^2}^2
		&= - \int_{\bbR^2} \partial^\alpha [(v \cdot \nabla)v] \cdot \partial^\alpha v + \partial^\alpha [(v\cdot \nb)B] \cdot \partial^\alpha B \,\mathrm{d}x\\ 
		&\quad+ \int_{\bbR^2} \partial^\alpha [(B \cdot \nabla)B] \cdot \partial^\alpha v + \partial^\alpha[(B\cdot\nb)v] \cdot \partial^\alpha B \,\mathrm{d}x
		\end{split}
	\end{equation}
	where we used integration by parts similarly such as the cancellation property
	$$\int_{\bbR^2} \partial^{\alpha} \partial_1B \cdot \partial^{\alpha} v \,\mathrm{d}x =- \int_{\bbR^2} \partial^{\alpha} \partial_1v \cdot \partial^{\alpha} B \,\mathrm{d}x.$$
	The first term involving $v$ and the derivatives of $v$ only in \eqref{energy_alpha} can be bounded as
	\begin{align*}
		\left| \int_{\bbR^2} \partial^\alpha (v \cdot \nabla)v \cdot \partial^\alpha v \,\mathrm{d}x\right| = \left|\int_{\bbR^2} \left( \partial^\alpha (v \cdot \nabla)v - (v \cdot \nabla)\partial^\alpha v \right) \cdot \partial^\alpha v \,\mathrm{d}x\right| \leq C\|v\|_{H^m}^3.
	\end{align*}
	We used the fact $\int (v \cdot \nabla)f \cdot f = 0$ for the equality, and the calculus inequality combined with the Sobolev embedding $L^{\infty}(\bbR^2) \hookrightarrow H^s(\bbR^2)$ for $s>1$ for the inequality.
	To estimate the other integrals, we consider the case $\partial^{\alpha} \neq \partial_2^{|\alpha|}$ first. Similarly with the above, we estimate the second term on the right-hand side of \eqref{energy_alpha} as
	\begin{equation*}
		\left| \int_{\bbR^2} \partial^\alpha (v \cdot \nabla)B \cdot \partial^\alpha B \,\mathrm{d}x \right| \leq C \| v \|_{H^m} \| B \|_{H^m} \| \partial_1 B \|_{H^{m-1}}.
	\end{equation*}
	Since integration by parts yields
	\begin{equation}\label{cancel_est}
		\int_{\bbR^2} (B \cdot \nabla)\partial^\alpha B \cdot \partial^\alpha v \,\mathrm{d}x + \int_{\bbR^2} (B \cdot \nabla)\partial^\alpha v \cdot \partial^\alpha B \,\mathrm{d}x = 0,
	\end{equation}
	we can infer that the other two integrals in \eqref{energy_alpha} can be estimated together as
	\begin{gather*}
		\left| \int_{\bbR^2} \partial^\alpha (B \cdot \nabla)B \cdot \partial^\alpha v \,\mathrm{d}x + \int_{\bbR^2} \partial^\alpha (B \cdot \nabla)v \cdot \partial^\alpha B \,\mathrm{d}x \right| \leq C \| v \|_{H^m} \| B \|_{H^m} \| \partial_1 B \|_{H^{m-1}}.
	\end{gather*}
	Now we treat the case $\partial^{\alpha}=\partial_2^{|\alpha|}$. Observe that
	\begin{equation*}
		\int_{\bbR^2} \partial^\alpha (v \cdot \nabla)B \cdot \partial^\alpha B \,\mathrm{d}x = \int_{\bbR^2} \partial_2^{|\alpha|-1} (\partial_2v \cdot \nabla)B \cdot \partial_2^{|\alpha|} B \,\mathrm{d}x.
	\end{equation*}
	The divergence-free condition implies $(\partial_2v \cdot \nabla)B = \partial_2v_1\partial_1B - \partial_1v_1 \partial_2B$ so that we have
	\begin{equation}\label{Bv_est}
		\int_{\bbR^2} \partial^\alpha (v \cdot \nabla)B \cdot \partial^\alpha B \,\mathrm{d}x = \int_{\bbR^2} \partial_2^{|\alpha|-1} (\partial_2v_1\partial_1B) \cdot \partial_2^{|\alpha|} B \,\mathrm{d}x - \int_{\bbR^2} \partial_2^{|\alpha|-1} (\partial_1v_1 \partial_2B) \cdot \partial_2^{|\alpha|} B \,\mathrm{d}x.
	\end{equation}
	The first integral of the right-hand side is estimated as
	\begin{equation*}
		\left| \int_{\bbR^2} \partial_2^{|\alpha|-1} (\partial_2v_1\partial_1B) \cdot \partial_2^{|\alpha|} B \,\mathrm{d}x \right| \leq C \| v \|_{H^m} \| B \|_{H^m} \| \partial_1 B \|_{H^{m-1}}.
	\end{equation*}
	Noting that
	\begin{equation}\label{vB_est}
		- \int_{\bbR^2} \partial_2^{|\alpha|-1} (\partial_1v_1 \partial_2B) \cdot \partial_2^{|\alpha|} B \,\mathrm{d}x = - \int_{\bbR^2} \partial_1v_1 |\partial_2^{|\alpha|} B |^2 \,\mathrm{d}x - \int_{\bbR^2} \partial_2^{|\alpha|-2} (\partial_1 \partial_2 v_1 \partial_2B) \cdot \partial_2^{|\alpha|} B \,\mathrm{d}x,
	\end{equation}
	we perform integration by parts to the second integral in \eqref{Bv_est} twice to obtain
	\begin{equation*}
		\left| \int_{\bbR^2} \partial_2^{|\alpha|-2} (\partial_1 \partial_2 v_1 \partial_2B) \cdot \partial_2^{|\alpha|} B \,\mathrm{d}x \right| \leq C \| v \|_{H^m} \| B \|_{H^m} \| \partial_1 B \|_{H^{m-1}}.
	\end{equation*}
	This implies that
	\begin{equation*}
		\left| \int_{\bbR^2} \partial_2^{|\alpha|-1} (\partial_1v_1 \partial_2B) \cdot \partial_2^{|\alpha|} B \,\mathrm{d}x \right| \leq \| \partial_1v_1 \|_{L^\infty} \| \partial_2^{|\alpha|} B \|_{L^2}^2 + C \| v \|_{H^m} \| B \|_{H^m} \| \partial_1 B \|_{H^{m-1}},
	\end{equation*}
	which leads to
\begin{equation*}
		\left| \int_{\bbR^2} \partial^\alpha (v \cdot \nabla)B \cdot \partial^\alpha B \,\mathrm{d}x \right| \leq \| \partial_1v_1 \|_{L^\infty} \| \partial_2^{|\alpha|} B \|_{L^2}^2 + C \| v \|_{H^m} \| B \|_{H^m} \| \partial_1 B \|_{H^{m-1}}.
	\end{equation*}
	By the use of \eqref{cancel_est}, it suffices to estimate the integral
	\begin{gather*}
		\int_{\bbR^2} \partial^\alpha (B \cdot \nabla)(B \cdot \partial^\alpha v +v \cdot \partial^\alpha B) \,\mathrm{d}x = \int_{\bbR^2} \partial_2^{|\alpha|-1} (\partial_2B \cdot \nabla)(B \cdot \partial_2^{|\alpha|} v + v \cdot \partial_2^{|\alpha|} B) \,\mathrm{d}x.
	\end{gather*}
	We can estimate the right-hand side by 
	$ C \| v \|_{H^m} \| B \|_{H^m} \| \partial_1 B \|_{H^{m-1}}$ except for the term that would be estimated as
	\begin{equation*}
		\left| \int_{\bbR^2} \partial_2^{|\alpha|-1} (\partial_2B \cdot \nabla)v_1 \cdot \partial_2^{|\alpha|} B_1 \,\mathrm{d}x \right| \leq \| \partial_1v_1 \|_{L^\infty} \| \partial_2^{|\alpha|} B \|_{L^2}^2 + C \| v \|_{H^m} \| B \|_{H^m} \| \partial_1 B \|_{H^{m-1}}
	\end{equation*}
	thanks to the relation $(\partial_2B \cdot \nabla)v_1 = \partial_2B_1 \partial_1 v_1 - \partial_1B_1 \partial_2 v_1$ from $\operatorname{div}B=0$. It follows that
	\begin{gather*}
		\left| \int_{\bbR^2} \partial^\alpha (B \cdot \nabla)(B \cdot \partial^\alpha v + v \cdot \partial^\alpha B) \,\mathrm{d}x \right| \leq \| \partial_1v_1 \|_{L^\infty} \| \partial_2^{|\alpha|} B \|_{L^2}^2 + C \| v \|_{H^m} \| B \|_{H^m} \| \partial_1 B \|_{H^{m-1}}.
	\end{gather*}
	Combining all the above estimates, we get
	\begin{equation}\label{Hm_est}
	\begin{gathered}
		\frac 12 \frac {\mathrm{d}}{\mathrm{d}t} E(t)^2 + \| v(t) \|_{H^m}^2 \\
		\leq 2\| \partial_1v(t) \|_{L^\infty} \| B(t) \|_{H^m}^2 + C \| v(t) \|_{H^m} \| \partial_1 B(t) \|_{H^{m-1}} \| B(t) \|_{H^m} + C \| v(t) \|_{H^m}^3.
	\end{gathered}
	\end{equation}
	To involve $A(t)$ as stated in \eqref{Em_ineq}, recalling the definition \eqref{A_def}, let $\alpha$ satisfy $0 \leq |\alpha| \leq m-1$. From the equations \eqref{mhd_eq}, we have
	\begin{gather*}
		\int_{\bbR^2} \partial_t \partial^\alpha v \cdot \partial_1 \partial^\alpha B \,\mathrm{d}x + \int_{\bbR^2} \partial^\alpha v \cdot \partial_1 \partial^\alpha B \,\mathrm{d}x + \int_{\bbR^2} \partial^\alpha (v \cdot \nabla)v \cdot \partial_1 \partial^\alpha B \,\mathrm{d}x \\
		- \int_{\bbR^2} \partial^\alpha (B \cdot \nabla)B \cdot \partial_1 \partial^\alpha B \,\mathrm{d}x = \int_{\bbR^2} |\partial_1\partial^\alpha B|^2 \,\mathrm{d}x
	\end{gather*}
	and
	\begin{equation*}
		\int_{\bbR^2} \partial_t \partial^\alpha B \cdot \partial_1 \partial^\alpha v \,\mathrm{d}x + \int_{\bbR^2} \partial^\alpha (v \cdot \nabla)B \cdot \partial_1 \partial^\alpha v \,\mathrm{d}x - \int_{\bbR^2} \partial^\alpha (B \cdot \nabla)v \cdot \partial_1 \partial^\alpha v \,\mathrm{d}x = \int_{\bbR^2} |\partial_1\partial^\alpha v|^2 \,\mathrm{d}x.
	\end{equation*}
	Substracting the frist equality from the second equality yields
	\begin{gather*}
		\frac {\mathrm{d}}{\mathrm{d}t} \int_{\bbR^2} \partial^\alpha B \cdot \partial_1 \partial^\alpha v \,\mathrm{d}x - \int_{\bbR^2} \partial^\alpha (v \cdot \nabla)v \cdot \partial_1 \partial^\alpha B \,\mathrm{d}x + \int_{\bbR^2} \partial^\alpha (B \cdot \nabla)B \cdot \partial_1 \partial^\alpha B \,\mathrm{d}x \\
		+ \int_{\bbR^2} \partial^\alpha (v \cdot \nabla)B \cdot \partial_1 \partial^\alpha v \,\mathrm{d}x - \int_{\bbR^2} \partial^\alpha (B \cdot \nabla)v \cdot \partial_1 \partial^\alpha v \,\mathrm{d}x \\
		= \int_{\bbR^2} |\partial_1\partial^\alpha v|^2 \,\mathrm{d}x + \int_{\bbR^2} \partial^\alpha v \cdot \partial_1 \partial^\alpha B \,\mathrm{d}x - \int_{\bbR^2} |\partial_1\partial^\alpha B|^2 \,\mathrm{d}x.
	\end{gather*}
	We estimate the integrals on the left-hand side first. We clearly see that
	\begin{gather*}
		\left| \int_{\bbR^2} \partial^\alpha (v \cdot \nabla)v \cdot \partial_1 \partial^\alpha B \,\mathrm{d}x \right| + \left| \int_{\bbR^2} \partial^\alpha (v \cdot \nabla)B \cdot \partial_1 \partial^\alpha v \,\mathrm{d}x \right| + \left| \int_{\bbR^2} \partial^\alpha (B \cdot \nabla)v \cdot \partial_1 \partial^\alpha v \,\mathrm{d}x \right| \\
		\leq C \| v \|_{H^m}^2 \| B \|_{H^m}.
	\end{gather*}
	In the case of $\partial^\alpha \neq \partial_2^{|\alpha|}$, simply we obtain
	\begin{equation*}
		\left| \int_{\bbR^2} \partial^\alpha (B \cdot \nabla)B \cdot \partial_1 \partial^\alpha B \,\mathrm{d}x \right| \leq C \| B \|_{H^m} \| \partial_1 B \|_{H^{m-1}}^2,
	\end{equation*}
	and for $\partial^\alpha = \partial_2^{|\alpha|}$, we can use the calculus inequality to get
	\begin{equation*}
		\left| \int_{\bbR^2} \partial^\alpha (B \cdot \nabla)B \cdot \partial_1 \partial^\alpha B \,\mathrm{d}x \right| \leq \| B_2 \|_{L^\infty} \| \partial_2^{|\alpha|+1} B \|_{L^2} \| \partial_1 \partial_2^{|\alpha|} B \|_{L^2} + C \| B \|_{H^m} \| \partial_1 B \|_{H^{m-1}}^2.
	\end{equation*}
	Meanwhile, the integrals on the right-hand side are controlled as
	\begin{gather*}
		\int_{\bbR^2} |\partial_1\partial^\alpha v|^2 \,\mathrm{d}x + \int_{\bbR^2} \partial^\alpha v \cdot \partial_1 \partial^\alpha B \,\mathrm{d}x - \int_{\bbR^2} |\partial_1\partial^\alpha B|^2 \,\mathrm{d}x \\
		\leq \int_{\bbR^2} |\partial_1\partial^\alpha v|^2 \,\mathrm{d}x + \frac 12 \int_{\bbR^2} |\partial^\alpha v|^2 \,\mathrm{d}x - \frac 12 \int_{\bbR^2} |\partial_1\partial^\alpha B|^2 \,\mathrm{d}x
	\end{gather*}
	and so we deduce that
	\begin{gather*}
		\frac {\mathrm{d}}{\mathrm{d}t} A_m(t) - \| B_2 \|_{L^\infty} \| B \|_{H^m} \| \partial_1 B \|_{H^{m-1}} - C \| B \|_{H^m} \| \partial_1 B \|_{H^{m-1}}^2 - C \| v \|_{H^m}^2 \| B \|_{H^m} \\
		\leq \frac 32 \| v \|_{H^m}^2 - \frac 12 \| \partial_1 B \|_{H^{m-1}}^2.
	\end{gather*}
	Multiplying \eqref{Hm_est} by $2$ and applying the just above inequality, \eqref{Em_ineq} is established as desired. This finishes the proof.
\end{proof}

\section{Proof of Theorem~\ref{thm1}: Global existence}\label{sec_thm}
In this section, we prove the existence of the unique global-in-time classical solution to \eqref{mhd_eq} for sufficiently small initial data. To this end, we perform a standard continuity argument based on the energy inequality obtained in Proposition~\ref{prop_eng}. To bound the right-hand side of \eqref{Em_ineq}, essentially we need the certain time integrability of the key terms as
\begin{equation}\label{purp}
	\| \partial_1 v(t) \|_{L^\infty(\bbR^2)} \in L^1(0,\infty) \qquad \mbox{and} \qquad \| B_2(t) \|_{L^\infty(\bbR^2)} \in L^2(0,\infty).
\end{equation}
To achieve that, we exploit the finer structure of the anisotropic linear propagator; we utilize the estimates in Lemma~\ref{lem_note}. The second integrability condition in \eqref{purp} is obtained in Proposition \ref{prop_B2}, being bounded below mainly by the other key quantity $\|\wwhat{\partial_1 v}\|_{L_T^1 L^1}$. This reduces the whole proof to estimating $\|\wwhat{\partial_1 v}\|_{L_T^1 L^1}$, which requires a sophisticated treatment - indeed, it is controlled by \emph{itself} multiplied by a factor that can be made small. See Proposition \ref{prop_vderiv}. 

Considering \eqref{purp}, we naturally define the target quantity $G$ by
\begin{equation}\label{G}
G(T):= \left( \sup_{t\in[0,T]}E(t)^2 + \int_{0}^{T}\left(\|v\|_{H^m}^2+\|\partial_1 B\|_{H^{m-1}}^2\right)\,\mathrm{d}t \right)^{\frac{1}{2}}.
\end{equation}
One may recall the definition of $E(t)$ in \eqref{E_def}. At the final step of the proof of the global existence, we use a standard continuity argument to show that $G(T)$ stays small once it was small. Then the solution can be extended for all time by Proposition \ref{loc_prop}. To execute such continuity argument, thanks to Proposition \ref{prop_eng}, it suffices to control the two key terms $\|B_2\|_{L^\infty}$ and $\|\partial_1 v\|_{L^\infty}$ in terms of the quantity $G(T)$, which is the main content of this section.

\begin{proposition}\label{prop_v} Let $m\geq 4.$ Assume that $(v,B)\in C([0,\infty);H^m(\bbR^2))$ is a classical solution to \eqref{mhd_eq}. Then there exists a constant $C>0$ such that
	\begin{equation}\label{3.1_v}
\|\widehat{v}\|_{L_T^{\frac43}L^1} \lesssim E(0) + G(T)^2 + \| \widehat{B}_2 \|_{L^2_T L^1} + \|\widehat{\nb B}_2\|_{L_T^{\frac43}L^1},
\end{equation}
	\begin{equation}\label{3.1_v2_1}
		\| \frac {|\widehat{v}_2|}{\sqrt{|\xi_1|}} \|_{L_T^{\frac 43} L^1} \lesssim E(0) + \| \frac {|\widehat{B}_0|}{|\xi|} \|_{L^1} + G(T)^2 + G(T)\left(\| \widehat{v} \|_{L_T^{\frac 43} L^1} + \| \widehat{\nabla B}_2 \|_{L_T^{\frac 43} L^1} \right),
	\end{equation}
	\begin{equation}\label{3.1_v2_2}
		\| \frac {|\widehat{v}_2|}{\sqrt{|\xi_1|}} \|_{L^2(0,T;L^2_{\xi_1}L^1_{\xi_2})} \lesssim  E(0) + \| \frac {\sqrt{|\xi_1|}}{|\xi|}|\widehat{\bfu}_0| \|_{L^2_{\xi_1}L^1_{\xi_2}} + G(T)^2 + G(T) \| \widehat{\nabla B}_2 \|_{L^{\frac{4}{3}_TL^1}}.
	\end{equation}
\end{proposition}
We provide a proof of this proposition in Appendix.

\begin{proposition}\label{prop_vderiv}
	Let $m \geq 4$. Assume that $(v,B)\in C([0,\infty);H^m(\bbR^2))$ is a classical solution to \eqref{mhd_eq}. Then there holds
	\begin{equation*}
\|\widehat{\partial_1 v}\|_{L_T^1 L^1} \lesssim E(0)+G(T)^2 +G(T)\left( \|\widehat{v}\|_{L_T^{\frac{4}{3}} L^1} +  \|\widehat{\partial_1 v}\|_{L_T^1 L^1} + \|\widehat{\partial_1 B}_2\|_{L_T^1 L^1} + \|\widehat{B}_2\|_{L_T^2 L^1}\right)
\end{equation*}
for all $T \geq 0$.
\end{proposition}
\begin{proof}
	From \eqref{df_u} and $\big| \widehat{\partial_1 v} \big| \lesssim \big| \widehat{\nabla v}_2 \big| \lesssim \big| i\xi \sum_{\gamma=\pm} \langle \widehat{\mathbf{u}},\mathbf{a}_{\gamma}^2 \rangle \langle \mathbf{b}_{\gamma}^2, e_2 \rangle \big|$, we have
	\begin{equation*}
		\|\widehat{\partial_1 v}\|_{L_T^1 L^1} \lesssim \left\| \int_{\bbR^2} \big|\xi \sum_{\gamma=\pm} \langle \wwhat{\mathbf{u}},\mathbf{a}_{\gamma}^2 \rangle \langle \mathbf{b}_{\gamma}^2, e_2 \rangle \big| \,\mathrm{d}\xi \right\|_{L^1(0,T)} \leq  I(\bbR^2) + J(\bbR^2) + K(\bbR^2),
	\end{equation*}
	where for any simply connected smooth domain $Q\subseteq \bbR^2$ we set 
	\begin{align*}
	I(Q) &:= \left\| \int_{Q} \big|\sum_{\gamma=\pm} e^{-\lambda_{\gamma}t} \xi \langle  \widehat{\mathbf{u}}_0(\xi),\mathbf{a}_{\gamma}^2(\xi) \rangle \langle \mathbf{b}_{\gamma}^2, e_2 \rangle \big| \,\mathrm{d}\xi \right\|_{L^1(0,T)}, \\
		J(Q) &:= \left\| \int_{Q}\int_0^t \big| \sum_{\gamma=\pm} e^{-\lambda_{\gamma}(\xi)(t-\tau)} \xi \langle \mathbf{P}(\xi)\wwhat{\partial_1 (v\cdot \nabla)\mathbf{u}}(\xi,\tau),\mathbf{a}_{\gamma}^2(\xi) \rangle \langle \mathbf{b}_{\gamma}^2, e_2 \rangle \big| \, \mathrm{d}\tau\mathrm{d}\xi \right\|_{L^1(0,T)}, \\
		K(Q) &:= \left\| \int_{Q}\int_0^t \big| \sum_{\gamma=\pm} e^{-\lambda_{\gamma}(\xi)(t-\tau)} \xi \langle \mathbf{P}(\xi)  \wwhat{\partial_1 (B\cdot \nabla)\widetilde{\mathbf{u}}}(\xi,\tau),\mathbf{a}_{\gamma}^2(\xi) \rangle \langle \mathbf{b}_{\gamma}^2, e_2 \rangle \big| \, \mathrm{d}\tau\mathrm{d}\xi \right\|_{L^1(0,T)}.
	\end{align*}
When one tries to directly bound the above quantities, the key difficulty lies in that $|\mathbf{b}_{\gamma}^{j}|$ for any $(j,\gamma)\in\{1,2\}\times\{+,-\}$ is unbounded along the lines $\xi_1=\frac12$ and $\xi_1=-\frac12$.  Our trick is to leverage the \emph{anisotropy decomposition} suggested in \eqref{decom_est_1}, which is, \begin{equation*}\label{decomp}
    \sum_{\gamma=\pm} e^{-\lambda_{\gamma}t} \langle  \widehat{\mathbf{f}},\mathbf{a}_{\gamma}^2 \rangle \langle \mathbf{b}_{\gamma}^2, e_2 \rangle = (e^{-\lambda_{-}t} - e^{-\lambda_{+}t}) \langle  \widehat{\mathbf{f}},\mathbf{a}_{-}^2 \rangle \langle \mathbf{b}_{-}^2, e_2 \rangle + e^{-\lambda_{+}t} \langle \widehat{\mathbf{f}}, e_2 \rangle ,\quad \forall \widehat{\mathbf{f}} \in \bbC^4,
\end{equation*} so that we have  $I^{+}(Q)+I^{-}(Q)\lesssim \mathcal{I}^{+}(Q)+\mathcal{I}^{-}(Q),$
where $\mathcal{I}^{+}(Q)$ and $\mathcal{I}^{-}(Q)$ are defined by \begin{equation}\label{def_decompint}
    \begin{split}
    \mathcal{I}^{+}(Q) &= \left\| \int_{Q} \big|e^{-\lambda_{+}t} \xi \langle  \widehat{\mathbf{u}}_0(\xi),e_2 \rangle \big| \,\mathrm{d}\xi \right\|_{L^1(0,T)},\  \\
        \mathcal{I}^{-}(Q)&= \left\| \int_{Q} \big|(e^{-\lambda_{-}t}-e^{-\lambda_{+}t}) \xi \langle  \widehat{\mathbf{u}}_0(\xi),\mathbf{a}_{-}^2(\xi) \rangle \langle \mathbf{b}_{-}^2, e_2 \rangle \big| \,\mathrm{d}\xi \right\|_{L^1(0,T)}.
    \end{split}
\end{equation} By $\mathcal{J}^{\pm}(Q)$ and $\mathcal{K}^{\pm}(Q)$, respectively, we denote the quantities that correspond to $J(Q)$ and $J(Q)$ in light of the decomposition \eqref{decom_est_1}. This allows us to exploit the simple but powerful estimates in Lemma~\ref{lem_note}, which are valid even near the singular regions.

	We start by bounding $\mathcal{I}^{+}(\bbR^2)$. The relation $|e^{-\lambda_{+}(\xi)t}| \leq e^{-\frac t2}$ gives
	\begin{align*}
	\mathcal{I}^{+}(\bbR^2) \lesssim \left\| \int_{\bbR^2} \big|e^{-\lambda_{+}t} \langle \widehat{\partial_1 \mathbf{u}}_0, e_2 \rangle \big| \,\mathrm{d}\xi \right\|_{L^1(0,T)} \lesssim \| e^{-\frac t2} \|_{L^1(0,T)} \int_{\bbR^2} \big|\widehat{\partial_1 v}_2(0) \big| \,\mathrm{d}\xi \lesssim \| v_2(0) \|_{H^{m-1}}.
	\end{align*}
	To bound $\mathcal{I}^{-}(\bbR^2)$, we use \eqref{Omg_1} and \eqref{Omg_2} to obtain
	\begin{equation*}
		\mathcal{I}^{-}(\Omega_1 \cup \Omega_2) \lesssim \| e^{-\frac t2} \|_{L^1(0,T)} \int_{\bbR^2} \big| \widehat{\partial_1 \mathbf{u}}_0 \big| \,\mathrm{d}\xi \lesssim \| \mathbf{u}_0 \|_{H^{m-1}}.
	\end{equation*}
	Since \eqref{Omg_4} guarantees
the point-wise estimate \begin{equation*}
	\begin{aligned}
		\big|(e^{-\lambda_{-}t} - e^{-\lambda_{+}t}) \xi \langle  \widehat{\mathbf{u}}_0,\mathbf{a}_{-}^2 \rangle \langle \mathbf{b}_{-}^2, e_2 \rangle \big| &\lesssim e^{-\xi_1^2t} |\xi| \left( \big|\wwhat{ \partial_1^2 v}_2(0) \big| + \big| \wwhat{ \partial_1B}_2(0) \big| \right) \\
		&\lesssim e^{-\xi_1^2t} \left( |\xi_1|^3\big| \wwhat{ v}_0 \big| + |\xi_1|^2\big| \wwhat{ B}_0 \big| \right)
	\end{aligned}
	\end{equation*}
	for any $\xi \in \Omega_3$, we also have
	\begin{equation*}
	\mathcal{I}^{-}(\Omega_3) \lesssim\bigg\| \int_{\bbR^2} \big|\xi_1^2 e^{-\xi_1^2t} \big| \big(\big|\widehat{\partial_1 v}_0 \big| + \big| \wwhat{ B}_0 \big| \big) \,\mathrm{d}\xi \bigg\|_{L^1(0,T)} \\
		 \lesssim  \| v_0 \|_{H^{m-1}} + \| B_0 \|_{H^{m-2}}.
	\end{equation*}
	This leads to
	\begin{equation*}
		I(\bbR^2)\lesssim \sum_{\gamma=\pm}\mathcal{I}^{\gamma}(\bbR^2) \lesssim  \| \mathbf{u}_0 \|_{H^m}.
	\end{equation*}
	
	With the definition of $\mathbf{P}(\xi)$ and Young's convolution inequality, we have
	\begin{equation}\label{vderiv_est_1}
	\begin{split}
		\mathcal{J}^{+}(\bbR^2) &\lesssim\left\| \int_{\bbR^2}\int_0^t e^{-\frac {t-\tau}2} \big| \wwhat{\partial_1 (v\cdot \nabla)\mathbf{u}}(\xi,\tau) \big| \, \mathrm{d}\tau\mathrm{d}\xi \right\|_{L^1(0,T)} \\
		& \lesssim \left\| \int_{\bbR^2}\big| \wwhat{(\partial_1 v\cdot \nabla)\mathbf{u}}(\xi,\tau) \big| \, \mathrm{d}\xi \right\|_{L^1(0,T)} + \left\| \int_{\bbR^2} \big| \wwhat{(v\cdot \nabla)\partial_1 \mathbf{u}}(\xi,\tau) \big| \, \mathrm{d}\xi \right\|_{L^1(0,T)} \\
		&\lesssim  G(T) \int_0^T \int_{\bbR^2} \big| \widehat{\partial_1 v}(\xi,t) \big| \,\ud \xi \ud t + G(T)^2.
	\end{split}
	\end{equation}
	Similarly, with \eqref{Omg_1} and \eqref{Omg_2}, we have
	\begin{align*}
		&\mathcal{J}^{-}(\Omega_1\cup\Omega_2) \lesssim \left\| \int_{\bbR^2}\int_0^t e^{-\frac {t-\tau}2} \big| \wwhat{\partial_1 (v\cdot \nabla) \mathbf{u}}(\xi,\tau) \big| \, \mathrm{d}\tau\mathrm{d}\xi \right\|_{L^1(0,T)} \\
		&\hphantom{\qquad\qquad\qquad} \lesssim  G(T) \int_0^T \int_{\bbR^2} \big| \widehat{\partial_1 v}(\xi,t) \big| \,\ud \xi \ud t + G(T)^2.
	\end{align*}
	Using \eqref{Omg_4} with $|\xi_1| \leq 1$ and the definition of $\mathbf{P}(\xi)$ gives
	\begin{equation*}
	\begin{aligned}
		\mathcal{J}^{-}(\Omega_3) &\lesssim \left\| \int_{\bbR^2}\int_0^t \big| \xi_1^2 e^{-\xi_1^2(t-\tau)} \wwhat{(v\cdot \nabla) \mathbf{u}}(\xi,\tau) \big| \, \mathrm{d}\tau\mathrm{d}\xi \right\|_{L^1(0,T)} \\
		&\hphantom{\qquad\qquad}\lesssim \int_0^T\int_{\bbR^2} \big| \wwhat{v_1\partial_1 \mathbf{u}}(\xi,t) \big| \, \mathrm{d}\xi \mathrm{d}t +  \int_0^T\int_{\bbR^2} \big| \wwhat{v_2 \partial_2 \mathbf{u}}(\xi,t) \big| \, \mathrm{d}\xi \mathrm{d}t \\
		&\hphantom{\qquad\qquad}\lesssim G(T)^2 + \| \widehat{\partial_2 B} \|_{L^4(0,T;L^1)} \| \widehat{v}_2 \|_{L^{\frac{4}{3}}(0,T;L^1)}.
	\end{aligned}
	\end{equation*}
	Therefore we obtain
	\begin{equation*}
		J(\bbR^2)\lesssim \sum_{\gamma=\pm}\mathcal{J}^{\gamma}(\bbR^2) \lesssim G(T)^2 + G(T) \int_0^T \int_{\bbR^2} \big| \widehat{\partial_1 v}(\xi,t) \big| \,\ud \xi \ud t + G(T)\| \widehat{v} \|_{L^{\frac{4}{3}}(0,T;L^1)}.
	\end{equation*}
	Now we treat $\mathcal{K}^{\pm}(\bbR^2)$. We immediately see that
	\begin{equation}\label{eq_K}
		\mathcal{K}^{+}(\bbR^2) \lesssim   \left\| \int_{\bbR^2} \big| \wwhat{(\partial_1 B \cdot \nabla) \mathbf{u}}(\xi,t) \big| \,\mathrm{d}\xi \right\|_{L^1(0,T)} + \left\| \int_{\bbR^2} \big| \wwhat{(B \cdot \nabla) \partial_1 \mathbf{u}}(\xi,t) \big| \,\mathrm{d}\xi \right\|_{L^1(0,T)}.
	\end{equation}
	The first term in the above can be estimated as
	\begin{equation*}
		\left\| \int_{\bbR^2} \big| \wwhat{(\partial_1 B \cdot \nabla) \mathbf{u}}(\xi,t) \big| \,\mathrm{d}\xi \right\|_{L^1(0,T)} \leq CG(T)^2 + G(T) \| \partial_1 B_2 \|_{L^1(0,T;L^1)}.
	\end{equation*}
	To estimate the second term in \eqref{eq_K}, we observe that
	\begin{equation*}
 \begin{split}
		\left\| \int_{\bbR^2} \big| \wwhat{(B \cdot \nabla) \partial_1 \mathbf{u}}(\xi,t) \big| \,\mathrm{d}\xi \right\|_{L^1(0,T)}
		\leq &\|\wwhat{B_1 \partial_1^2 \mathbf{u}}\|_{L_T^1 L^1} + \|\wwhat{B_2 \partial_1\partial_2 \mathbf{u}}\|_{L_T^1 L^1} \\
		\leq &\|\wwhat{B_1 \partial_1^2 v}\|_{L_T^1 L^1}+ \|\wwhat{B_1 \partial_1^2 B}\|_{L_T^1 L^1} + C\| \widehat{B}_2 \|_{L_T^2 L^1} \| \partial_1 \bfu \|_{L_T^2 H^{m-1}}.
	\end{split}
 \end{equation*}
	From the interpolation inequalities $\|\widehat{\partial_1^2v}\|_{L^1}  \lesssim \| \widehat{\partial_1v} \|_{L^1}^{\frac 12} \| v \|_{H^4}^{\frac 12}$, $\|\widehat{\partial_1^2B_1}\|_{L^1} \lesssim \| \widehat{\partial_1B}_2 \|_{L^1}^{\frac 12} \| \partial_1 B \|_{H^3}^{\frac 12}$, and $\|\widehat{\partial_1^2B}_2\|_{L^1} \,\ud \xi \lesssim \| \widehat{\partial_1B}_2 \|_{L^1}^{\frac 12} \| \partial_1 B \|_{H^3}^{\frac 12}$, among which we used $\operatorname{div}B=0$, it follows that
	\begin{equation*}
	\begin{gathered}
		\left\| \int_{\bbR^2} \big| \wwhat{(B \cdot \nabla) \partial_1 \mathbf{u}}(\xi,t) \big| \,\mathrm{d}\xi \right\|_{L^1(0,T)} \lesssim G(T)( \| \widehat{\partial_1 v} \|_{L_T^1 L^1} +  \| \widehat{\partial_1B}_2 \|_{L_T^1 L^1} + \| \widehat{B}_2 \|_{L_T^2 L^1)}) + G(T)^2.
	\end{gathered}
	\end{equation*}
	It suffices to calculate $\mathcal{K}^{-}(\bbR^2).$ Since \eqref{Omg_1} and \eqref{Omg_2} yield the point-wise inequality
	\begin{equation*}
		\big| (e^{-\lambda_{-}(\xi)(t-\tau)} - e^{-\lambda_{+}(\xi)(t-\tau)}) \xi \langle \mathbf{P}(\xi)\wwhat{(B\cdot \nabla)\widetilde{\mathbf{u}}}, \mathbf{a}_{-}^2 \rangle \langle \mathbf{b}_{-}^2,e_2 \rangle \big| \lesssim e^{-\frac {t-\tau}2}\big| \wwhat{\partial_1(B\cdot \nabla) \bfu} \big|, \qquad \xi \in\Omega_1 \cup \Omega_2,
	\end{equation*}
	we similarly obtain
	\begin{align*}
		\mathcal{K}^{-}(\Omega_1\cup\Omega_2) &\lesssim 
	 \left\| \int_{\bbR^2}\int_0^t e^{-\frac {t-\tau}2} \big| \wwhat{\partial_1 (B \cdot \nabla) \mathbf{u}}(\xi,\tau) \big| \, \mathrm{d}\tau\mathrm{d}\xi \right\|_{L^1(0,T)} \\
		& \lesssim G(T)^2 + G(T) (\| \widehat{\partial_1 v} \|_{L^1(0,T;L^1)} +  \| \widehat{\partial_1B}_2 \|_{L^1(0,T;L^1)} + \| \widehat{B}_2 \|_{L^2(0,T;L^1)}).
	\end{align*}
	On the other hand, \eqref{Omg_3} gives
	\begin{equation*}
		\big| \xi \langle \mathbf{P}(\xi)\wwhat{(B\cdot \nabla)\widetilde{\mathbf{u}}}, \mathbf{a}_{-}^2 \rangle \langle \mathbf{b}_{-}^2,e_2 \rangle \big|\leq C \big| \wwhat{\partial_1^2(B\cdot \nabla) v} \big| + C \big| \wwhat{\partial_1^3(B\cdot \nabla) B} \big|, \qquad \xi \in\Omega_3.
	\end{equation*}
	We can deduce that
	\begin{equation*}
	\begin{split}
 \mathcal{K}^{-}(\Omega_3) &\lesssim \left\| \int_{\Omega_3} \big| \wwhat{(B\cdot \nabla) v} \big| \,\mathrm{d}\xi \right\|_{L^1(0,T)} + \left\| \int_{\Omega_3}\big| \wwhat{\partial_1(B\cdot \nabla) B} \big| \,\mathrm{d}\xi \right\|_{L^1(0,T)} \\
		&\lesssim G(T)^2 + G(T) (\| \widehat{\partial_1 v} \|_{L^1(0,T;L^1)} +  \| \widehat{\partial_1B}_2 \|_{L^1(0,T;L^1)} +  \| \widehat{B}_2 \|_{L^2(0,T;L^1)}).
	\end{split}
	\end{equation*}
	Thus, $\mathcal{K}^{\pm}(\bbR^2) \lesssim G(T)^2 +  G(T) (\| \widehat{\partial_1 v} \|_{L^1(0,T;L^1)} +  \| \widehat{\partial_1B}_2 \|_{L^1(0,T;L^1)} +  \| \widehat{B}_2 \|_{L^2(0,T;L^1)})$. Collecting all the above estimates, we complete the proof. 	
\end{proof}

\begin{proposition}\label{prop_Bderiv}
	Let $m > 3$ be an integer. Assume that $(v,B)\in C([0,\infty);H^m(\bbR^2))$ is a classical solution to \eqref{mhd_eq}. Then there exists a constant $C>0$ such that
	\begin{equation*}
	\begin{gathered}
		\|\widehat{\nabla B}_2\|_{L_T^{\frac43}L^1} \lesssim E(0) + G(T)^2 +  \|\frac{\big|\wwhat{B}_0 \big|}{\sqrt{|\xi_1|}}\|_{L^1} + G(T)\left(\| \wwhat{\partial_1 v} \|_{L_T^1 L^1}
		+ \| \frac {|\widehat{v}_2|}{\sqrt{|\xi_1|}} \|_{L_T^{\frac 43}L^1} +  \| \frac {|\widehat{v}_2|}{\sqrt{|\xi_1|}} \|_{L_T^2 L^2_{\xi_1}L^1_{\xi_2}}\right)
	\end{gathered}
	\end{equation*}
	for all $T \geq 0$.
\end{proposition}
\begin{proof}
    From \eqref{df_u} and $\widehat{\nabla B}_2 = -i\xi \sum_{\gamma=\pm} \langle  \widehat{\mathbf{u}},\mathbf{a}_{\gamma}^2 \rangle \langle \mathbf{b}_{\gamma}^2, e_4 \rangle$,
	we have
	\begin{equation*}
		\left\| \int_{\bbR^2} \big|\xi \sum_{\gamma=\pm} \langle \wwhat{ \mathbf{u}},\mathbf{a}_{\gamma}^2 \rangle \langle \mathbf{b}_{\gamma}^2, e_4 \rangle \big| \,\mathrm{d}\xi \right\|_{L^{\frac{4}{3}}(0,T)} \leq I(\bbR^2) + J(\bbR^2) + K(\bbR^2),
	\end{equation*}
	where for any simply connected smooth domain $Q\subseteq \bbR^2$ we set 
	\begin{align*}
    	I(Q) &:= \left\| \int_{Q} \big| \sum_{\gamma=\pm} e^{-\lambda_{\gamma}(\xi)t} \xi \langle  \widehat{\mathbf{u}}_0(\xi),\mathbf{a}_{\gamma}^2(\xi) \rangle \langle \mathbf{b}_{\gamma}^2, e_4 \rangle \big| \,\mathrm{d}\xi \right\|_{L^{\frac{4}{3}}(0,T)}, \\
		J(Q) &:= \left\| \int_{Q} \int_0^t \big| \sum_{\gamma=\pm} e^{-\lambda_{\gamma}(\xi)(t-\tau)} \xi \langle \mathbf{P}(\xi)\wwhat{(v\cdot \nabla)\mathbf{u}}(\xi,\tau),\mathbf{a}_{\gamma}^2(\xi) \rangle \langle \mathbf{b}_{\gamma}^2, e_4 \rangle \big| \, \mathrm{d}\tau \mathrm{d}\xi \right\|_{L^{\frac{4}{3}}(0,T)}, \\
		K(Q) &:= \left\| \int_{Q}\int_0^t \big| \sum_{\gamma=\pm} e^{-\lambda_{\gamma}(\xi)(t-\tau)} \xi \langle \mathbf{P}(\xi)  \wwhat{(B\cdot \nabla)\widetilde{\mathbf{u}}}(\xi,\tau),\mathbf{a}_{\gamma}^2(\xi) \rangle \langle \mathbf{b}_{\gamma}^2, e_4 \rangle \big| \, \mathrm{d}\tau\mathrm{d}\xi \right\|_{L^{\frac{4}{3}}(0,T)}.
	\end{align*}
	As we leveraged the anisotropy decomposition \eqref{decom_est_1} in proving Proposition~\ref{prop_B2}, we analogously recall the decomposition \eqref{decom_est_2}, that is,
 \begin{equation*}
     \sum_{\gamma=\pm} e^{-\lambda_{\gamma}t} \langle  \widehat{\mathbf{f}},\mathbf{a}_{\gamma}^2 \rangle \langle \mathbf{b}_{\gamma}^2, e_4 \rangle = (e^{-\lambda_{-}t} - e^{-\lambda_{+}t}) \langle  \widehat{\mathbf{f}},\mathbf{a}_{-}^2 \rangle \langle \mathbf{b}_{-}^2, e_4 \rangle + e^{-\lambda_{+}t} \langle \widehat{\mathbf{f}}, e_4 \rangle, \quad \forall \mathbf{f}\in\bbC^4.
 \end{equation*}
This would imply $I(Q) \lesssim \sum_{\gamma=\pm}\mathcal{I}^{\gamma}(Q) $, where $\mathcal{I}^{+}(Q)$ and $\mathcal{I}^{-}(Q)$ are defined as
\begin{equation*}
    \begin{split}
        \mathcal{I}^{+}(Q)&= \left\| \int_{Q} \big|e^{-\lambda_{+}(\xi)t} \xi \langle  \widehat{\mathbf{u}}_0(\xi),\mathbf{e_4}\rangle \big| \,\mathrm{d}\xi \right\|_{L^{\frac{4}{3}}(0,T)} \\
        \mathcal{I}^{-}(Q)&=  \left\| \int_{Q} \big|(e^{-\lambda_{-}t}-e^{-\lambda_{+}t}) \xi \langle  \widehat{\mathbf{u}}_0(\xi),\mathbf{a}_{-}^2(\xi) \rangle \langle \mathbf{b}_{-}^2, e_4 \rangle \big| \,\mathrm{d}\xi \right\|_{L^{\frac{4}{3}}(0,T)}
    \end{split}
\end{equation*}
Then the quantities $\mathcal{J}^{\pm}(Q)$ and $\mathcal{K}^{\pm}(Q)$, respectively, are defined correspondingly to $J(Q)$ and $K(Q)$ in view of the above decomposition \eqref{decom_est_2}.
 
 We estimate $I(\bbR^2)$ similarly with the proof of Proposition~\ref{prop_vderiv}, using \eqref{decom_est_2}. It is clear by \eqref{Omg_1} that
	\begin{equation*}
	\mathcal{I}^{+}(\bbR^2) \lesssim \| e^{-\frac t2} \|_{L^{\frac{4}{3}}(0,T)} \int_{\bbR^2} \big|\widehat{\nabla B}_2(0) \big| \,\mathrm{d}\xi \lesssim \| B_2(0) \|_{H^m}.
	\end{equation*}
	From \eqref{Omg_1} and \eqref{Omg_2}, 
	\begin{equation*}
	\mathcal{I}^{-}(\Omega_1\cup\Omega_2) \lesssim \left\| \int_{\bbR^2} e^{-\frac t2} \big| \widehat{\nabla \mathbf{u}}_0 \big| \,\mathrm{d}\xi \right\|_{L^{\frac{4}{3}}(0,T)} \lesssim \| \mathbf{u}_0 \|_{H^m}.
	\end{equation*}
	Combining $e^{-\lambda_{+}t} \leq e^{-\lambda_{-}t} \leq e^{-\xi_1^2t} $ with
\eqref{B2_initial}, we see that 
	\begin{align*}
		\mathcal{I}^{-}(\Omega_3) \lesssim \bigg\| \int_{\bbR^2} \big|\xi_1^{\frac 32} e^{-\xi_1^2t} \big| \big(|\xi_1|^{\frac 12} \big|\widehat{v}_0 \big| + |\xi_1|^{-\frac 12} \big| \wwhat{ B}_0 \big| \big) \,\mathrm{d}\xi \bigg\|_{L^{\frac{4}{3}}(0,T)} \lesssim \| v_0 \|_{H^m} +  \int_{\bbR^2} \frac{\big|\wwhat{B}_0 \big|}{\sqrt{\xi_1}} \,\ud \xi.
	\end{align*}
	Thus, we obtain
	\begin{equation*}
		I(\bbR^2) \lesssim \sum_{\gamma=\pm} \mathcal{I}^{\gamma}(\bbR^2) \lesssim \| \mathbf{u}_0 \|_{H^m} +  \int_{\bbR^2} \frac{\big|\wwhat{B}_0 \big|}{\sqrt{\xi_1}} \,\ud \xi.
	\end{equation*}
	
	We can estimate $\mathcal{J}^{+}(\bbR^2)$ just as we did in \eqref{vderiv_est_1}. For any $m>3$, there holds
	\begin{align*}
	\mathcal{J}^{+}(\bbR^2) &\lesssim \left\| \int_{\bbR^2}\int_0^t e^{-\frac {t-\tau}2} \big| \wwhat{\partial_1 (v \cdot \nabla) \bfu} (\xi,\tau) \big| \, \mathrm{d}\tau\mathrm{d}\xi \right\|_{L^{\frac{4}{3}}(0,T)} \\
		&\lesssim \|\wwhat{(\partial_1v  \cdot \nabla) \mathbf{u}}\|_{L_T^1 L^1}+  \|\wwhat{(v  \cdot \nabla) \partial_1\mathbf{u}}\|_{L_T^1 L^1}\\
  &\lesssim G(T) \|\widehat{\partial_1 v}\|_{L_T^1 L^1} +  G(T)^2.
	\end{align*}
	Then we bound $\mathcal{J}^{-}(\Omega_1\cup\Omega_2)$ as
	\begin{align*}
		\mathcal{J}^{-}(\Omega_1\cup\Omega_2) &\lesssim \left\| \int_{\bbR^2}\int_0^t e^{-\frac {t-\tau}2} \big| \wwhat{\partial_1 (v\cdot \nabla)\mathbf{u}}(\xi,\tau) \big| \, \mathrm{d}\tau\mathrm{d}\xi \right\|_{L^{\frac{4}{3}}(0,T)} \\
        &\lesssim G(T) \int_0^T \int_{\bbR^2} \big| \widehat{\partial_1 v} \big| \,\ud \xi \ud t +  G(T)^2
	\end{align*}
 with the help of the anisotropic estimates \eqref{Omg_1} and \eqref{Omg_2}. To control $\mathcal{J}^{-}(\Omega_3)$, we simply use $(v \cdot \nabla) \bfu = v_1 \partial_1 \bfu + v_2 \partial_2 \bfu$ and then estimate separately the two integrals that correspond to $v_1 \partial_1 \bfu$ and $v_2 \partial_2 \bfu$. Let us call the integrals $\mathcal{J}_1^{-}(\Omega_3)$ and $\mathcal{J}_2^{-}(\Omega_3)$, respectively. By the estimate \eqref{Omg_3}, which is designed specifically for $\Omega_3$, and by the definition of $\mathbf{P}(\xi)$, we bound $\mathcal{J}_1^{-}(\Omega_3)$ as
	\begin{equation*}
	\begin{aligned}
		\mathcal{J}_1^{-}(\Omega_3) \lesssim \left\| \int_{\bbR^2}\int_0^t (t-\tau)^{-\frac 12} \big| e^{-\xi_1^2 \frac{t-\tau}2} \wwhat{v_1 \partial_1 \mathbf{u}}(\xi,\tau) \big| \, \mathrm{d}\tau\mathrm{d}\xi \right\|_{L^{\frac{4}{3}}(0,T)}.
	\end{aligned}
	\end{equation*}
	We extract additional algebraic temporal decay out of the anisotropic time decay $e^{-\xi_1^2 \frac{t-\tau}{2}}$ as
	\begin{align*}
		\int_{\bbR^2} \big| e^{-\xi_1^2 \frac{t-\tau}2} \wwhat{v_1 \partial_1 \mathbf{u}} \big| \, \mathrm{d}\xi &\leq (t-\tau)^{-\frac 12(1-\frac 1p)} \| \wwhat{v_1 \partial_1 \mathbf{u}} \|_{L^p_{\xi_1}L^1_{\xi_2}} \\
		&\leq (t-\tau)^{-\frac 12(1-\frac 1p)} \| \wwhat{v}_1 \|_{L^q_{\xi_1}L^1_{\xi_2}} \|  \wwhat{\partial_1 \mathbf{u}} \|_{L^2_{\xi_1}L^1_{\xi_2}} \\
		&\leq (t-\tau)^{-\frac 12(1-\frac 1p)} \| \wwhat{v}_1 \|_{L^1}^{\frac 2p}  \| \wwhat{v}_1 \|_{L^2_{\xi_1}L^1_{\xi_2}}^{1-\frac 2p}  \| \wwhat{\partial_1 \mathbf{u}} \|_{L^2_{\xi_1}L^1_{\xi_2}} ,
	\end{align*}
	where $1/p + 1 = 1/q + 1/2$ and $p \in (2,\infty)$. The above observation combined with Young's convolution inequality gives
	\begin{equation}\label{Bderiv_est_1}
	\begin{aligned}
		\mathcal{J}_1^{-}(\Omega_3) \lesssim \left\| \int_{\bbR^2}\int_0^t (t-\tau)^{-\frac 12} \big| e^{-\xi_1^2 \frac{t-\tau}2} \wwhat{v_1 \partial_1 \mathbf{u}}(\xi,\tau) \big| \, \mathrm{d}\tau\mathrm{d}\xi \right\|_{L^{\frac{4}{3}}(0,T)} \lesssim G(T)^2.
	\end{aligned}
	\end{equation}
	Similarly, the second term $\mathcal{J}_2^{-}(\Omega_3)$ is estimated as
	\begin{equation*}
	\begin{aligned}
		\mathcal{J}_2^{-}(\Omega_3)
		\lesssim \left\| \int_{\bbR^2}\int_0^t \big| \xi_1 e^{-\xi_1^2 (t-\tau)} \wwhat{v_2 \partial_2 \mathbf{u}}(\xi,\tau) \big| \, \mathrm{d}\tau\mathrm{d}\xi \right\|_{L^{\frac{4}{3}}(0,T)}.
	\end{aligned}
	\end{equation*}
During further estimation of $\mathcal{J}_2^{-}(\Omega_3)$, for simplicity, we may denote by $g$ the function $g:\bbR^2\to\bbR_{+}$ defined by \begin{equation}\label{def_g}
g(\xi)=\frac{1}{\sqrt{|\xi_1|}}.
\end{equation} Then we observe that
\begin{equation*}
		\big| \wwhat{v_2 \partial_2 \mathbf{u}}(\xi) \big| \leq \sqrt{|\xi_1|} \left((g|\widehat{v}_2|) * \big| \widehat{\partial_2 \bfu} \big| \right)(\xi)+ \left((g|\widehat{v}_2|) *  \frac{|\widehat{\partial_2 \bfu}|}{g}\right)(\xi).
	\end{equation*}
	Young's convolution inequality implies
	\begin{equation*}
	\begin{aligned}
		 \left\| \int_{\bbR^2}\int_0^t \big| \xi_1^{\frac 32} e^{-\xi_1^2 (t-\tau)} \left((g|\widehat{v}_2|)* \big| \widehat{\partial_2 \bfu}\big|\right) (\xi,\tau)|  \, \mathrm{d}\tau\mathrm{d}\xi \right\|_{L^{\frac{4}{3}}(0,T)} & \lesssim \left\| \| \frac {|\widehat{v}_2|}{\sqrt{|\xi_1|}}\|_{L^1} \| \widehat{\partial_2 \bfu} \|_{L^1} \right\|_{L^1(0,T)} \\
		& \lesssim  \| \frac {|\widehat{v}_2|}{\sqrt{|\xi_1|}}\|_{L_T^{\frac 43} L^1} G(T).
	\end{aligned}
	\end{equation*} 
	Since
	\begin{equation*}
	\begin{aligned}
		\int_{\bbR^2} \big| \xi_1 e^{-\xi_1^2 (t-\tau)} &\left((g|\widehat{v}_2|) *  \frac{|\widehat{\partial_2 \bfu}|}{g}\right)(\xi) \big| \, \mathrm{d}\xi \\ &\lesssim(t-\tau)^{-\frac 12(1-\frac 1p)} \| (g|\widehat{v}_2|) *  \frac{|\widehat{\partial_2 \bfu}|}{g} \|_{L^p_{\xi_1}L^1_{\xi_2}} \\
&\lesssim(t-\tau)^{-\frac 12(1-\frac 1p)} \| \frac {|\widehat{v}_2|}{\sqrt{|\xi_1|}} \|_{L^1}^{\frac 2p}  \| \frac {|\widehat{v}_2|}{\sqrt{|\xi_1|}} \|_{L^2_{\xi_1}L^1_{\xi_2}}^{1-\frac 2p}  \| \wwhat{\partial_1 \partial_2 \mathbf{u}} \|_{L^2_{\xi_1}L^1_{\xi_2}}^{\frac 12}  \| \wwhat{\partial_2 \mathbf{u}} \|_{L^2_{\xi_1}L^1_{\xi_2}}^{\frac 12},
	\end{aligned}
	\end{equation*}
	where $1/p + 1 = 1/q + 1/2$ and $p \in (2,\infty)$, we obtain
	\begin{equation*}
	\begin{aligned}
		& \left\| \int_{\bbR^2}\int_0^t \big| \xi_1 e^{-\xi_1^2 (t-\tau)} \frac {|\widehat{v}_2|}{\sqrt{|\xi_1|}} * \big| \sqrt{|\xi_1|} \widehat{\partial_2 \bfu} \big|(\xi,\tau) \big| \, \mathrm{d}\tau\mathrm{d}\xi \right\|_{L^{\frac{4}{3}}(0,T)}  \\
		&\hphantom{\qquad\qquad} \lesssim \left\| \int_0^t (t-\tau)^{-1+\frac 1{2p}} \| \frac {|\widehat{v}_2|}{\sqrt{|\xi_1|}} \|_{L^1}^{\frac 2p}  \| \frac {|\widehat{v}_2|}{\sqrt{|\xi_1|}} \|_{L^2_{\xi_1}L^1_{\xi_2}}^{1-\frac 2p}  \| \wwhat{\partial_1 \partial_2 \mathbf{u}} \|_{L^2_{\xi_1}L^1_{\xi_2}}^{\frac 12}  \| \wwhat{\partial_2 \mathbf{u}} \|_{L^2_{\xi_1}L^1_{\xi_2}}^{\frac 12} \,\ud \tau \right\|_{L^{\frac 43}(0,T)} \\
		&\hphantom{\qquad\qquad} \lesssim \left\| \| \frac {|\widehat{v}_2|}{\sqrt{|\xi_1|}} \|_{L^1}^{\frac 2p}  \| \frac {|\widehat{v}_2|}{\sqrt{|\xi_1|}} \|_{L^2_{\xi_1}L^1_{\xi_2}}^{1-\frac 2p}  \| \wwhat{\partial_1 \partial_2 \mathbf{u}} \|_{L^2_{\xi_1}L^1_{\xi_2}}^{\frac 12}  \| \wwhat{\partial_2 \mathbf{u}} \|_{L^2_{\xi_1}L^1_{\xi_2}}^{\frac 12} \right\|_{L^{\frac {4p}{3p + 2}}(0,T)} \\
		&\hphantom{\qquad\qquad} \lesssim \| \frac {|\widehat{v}_2|}{\sqrt{|\xi_1|}} \|_{L_T^{\frac 43} L^1}^{\frac 2p}  \| \frac {|\widehat{v}_2|}{\sqrt{|\xi_1|}} \|_{L^2(0,T;L^2_{\xi_1}L^1_{\xi_2})}^{1-\frac 2p}  G(T).
	\end{aligned}
	\end{equation*}
	Thus,
	\begin{equation*}
		\mathcal{J}^{\pm}(\bbR^2) \lesssim G(T)^2 + G(T)(\| \wwhat{v}_2 \|_{L^1(0,T;L^1)} + \| \frac {|\widehat{v}_2|}{\sqrt{|\xi_1|}} \|_{L_T^{\frac 43} L^1} +  \| \frac {|\widehat{v}_2|}{\sqrt{|\xi_1|}} \|_{L^2(0,T;L^2_{\xi_1}L^1_{\xi_2})}).
	\end{equation*}
It suffices to estimate $\mathcal{K}^{\pm}(\bbR^2).$ We see that
	\begin{align*}
		\mathcal{K}^{+}(\bbR^2) &\lesssim  \left\| \int_{\bbR^2}\int_0^t e^{-\frac {t-\tau}2} \big| \wwhat{\partial_1 (B \cdot \nabla) \widetilde{\bfu}} (\xi,\tau) \big| \, \mathrm{d}\tau\mathrm{d}\xi \right\|_{L^{\frac{4}{3}}(0,T)} \\
		&\lesssim  \left\| \int_{\bbR^2}\big| \wwhat{(\partial_1B \cdot \nabla)\bfu} \big| \, \mathrm{d}\xi \right\|_{L^{\frac 43}(0,T)} +  \left\| \int_{\bbR^2}\big| \wwhat{(B \cdot \nabla) \partial_1\bfu} \big| \, \mathrm{d}\xi \right\|_{L^{\frac 43}(0,T)} \\
		&\lesssim G(T)^2.
	\end{align*}
	Similarly, we have
	from \eqref{Omg_1}, \eqref{Omg_2}, and the definition of $\bfP(\xi)$ that
	\begin{align*}
		\mathcal{K}^{-}(\Omega_1\cup\Omega_2) \lesssim \left\| \int_{\bbR^2}\int_0^t e^{-\frac {t-\tau}2} \big| \wwhat{\partial_1 (B\cdot \nabla)\widetilde{\bfu}}(\xi,\tau) \big| \, \mathrm{d}\tau\mathrm{d}\xi \right\|_{L^{\frac{4}{3}}(0,T)}
        \lesssim G(T)^2.
	\end{align*}
	It only remains to bound $\mathcal{K}^{-}(\Omega_3)$. We deduce from \eqref{Omg_3} that\begin{equation}\label{eq_Kdecomp}
	    \big| \xi \langle \mathbf{P}(\xi)\wwhat{(B\cdot \nabla)\widetilde{\bfu}}(\xi,\tau),e_4 \rangle \langle \mathbf{b}_{-}^2,e_4 \rangle \big| \lesssim e^{-\xi_1^2t}\big(\big|\wwhat{\partial_1^2(B \cdot \nabla)B}\big| + \big|\wwhat{\partial_1 (B \cdot \nabla)v}\big|\big).\end{equation} This gives rise to the decomposition inequality $\mathcal{K}^{-}(\Omega_3) \lesssim \mathcal{K}_1^{-}(\Omega_3)+\mathcal{K}_2^{-}(\Omega_3)$
 where $\mathcal{K}_1^{-}(\Omega_3)$ and $\mathcal{K}_2^{-}(\Omega_3)$ stand for the corresponding quantities to the first and the second in the right-hand side of \eqref{eq_Kdecomp}, respectively. Using $\partial_1((B \cdot \nabla) v) = \partial_1(\partial_1 (B_1 v) - \partial_1 B_1 v + B_2 \partial_2 v)$, we estimate $\mathcal{K}_{2}^{-}(\Omega_3)$ term by term. Observe that
	\begin{equation*}
	\begin{aligned}
		\left\| \int_{\Omega_3} \int_0^t \big| e^{-\xi_1^2(t-\tau)} \wwhat{\partial_1^2 (B _1v)} (\xi,\tau) \big| \, \mathrm{d}\tau \mathrm{d}\xi \right\|_{L^{\frac{4}{3}}(0,T)} &\lesssim \left\| \int_{\bbR^2}\int_0^t \big| \xi_1^2 e^{-\xi_1^2 \frac{t-\tau}2} \wwhat{B_1 v}(\xi,\tau) \big| \, \mathrm{d}\tau\mathrm{d}\xi \right\|_{L^{\frac{4}{3}}(0,T)} \\
		&\lesssim \left\| \int_{\bbR^2} \big| \wwhat{B_1 v}(\xi,t) \big| \, \mathrm{d}\xi \right\|_{L^{\frac{4}{3}}(0,T)} \\
		&\lesssim G(T)^2.
	\end{aligned}
	\end{equation*}
	Since
	\begin{equation*}
	\begin{aligned}
		&\left\| \int_{\Omega_3} \int_0^t \big| e^{-\xi_1^2(t-\tau)} \wwhat{\partial_1 (\partial_1B _1 v)} (\xi,\tau) \big| \, \mathrm{d}\tau \mathrm{d}\xi \right\|_{L^{\frac{4}{3}}(0,T)} \lesssim \left\| \int_{\bbR^2}\int_0^t (t-\tau)^{-\frac 12} \big| e^{-\xi_1^2 \frac{t-\tau}2} \wwhat{\partial_1 B_1 v}(\xi,\tau) \big| \, \mathrm{d}\tau\mathrm{d}\xi \right\|_{L^{\frac{4}{3}}(0,T)},
	\end{aligned}
	\end{equation*}
	recalling \eqref{Bderiv_est_1}, we obtain
	\begin{equation*}
	\begin{aligned}
		&\left\| \int_{\Omega_3} \int_0^t \big| e^{-\xi_1^2(t-\tau)} \wwhat{\partial_1 (\partial_1B _1 v)} (\xi,\tau) \big| \, \mathrm{d}\tau \mathrm{d}\xi \right\|_{L^{\frac{4}{3}}(0,T)} \lesssim G(T)^2.
	\end{aligned}
	\end{equation*}
	We can see
	\begin{equation*}
	\begin{aligned}
		\left\| \int_{\Omega_3} \int_0^t \big| e^{-\xi_1^2(t-\tau)} \wwhat{\partial_1 (B_2\partial_2v)} (\xi,\tau) \big| \, \mathrm{d}\tau \mathrm{d}\xi \right\|_{L^{\frac{4}{3}}(0,T)} \lesssim \left\| \int_{\bbR^2}\int_0^t (t-\tau)^{-\frac 12} \big| e^{-\xi_1^2 \frac{t-\tau}2} \wwhat{B_2 \partial_2 v}(\xi,\tau) \big| \, \mathrm{d}\tau\mathrm{d}\xi \right\|_{L^{\frac{4}{3}}(0,T)}.
	\end{aligned}
	\end{equation*}
	Note that
		\begin{align*}
		\int_{\bbR^2} \big| e^{-\xi_1^2 \frac{t-\tau}2} \wwhat{B_2 \partial_2 v} \big| \, \mathrm{d}\xi &\lesssim (t-\tau)^{-\frac 12(1-\frac 1p)} \| \wwhat{B_2 \partial_2 v} \|_{L^p_{\xi_1}L^1_{\xi_2}} \\
		&\lesssim (t-\tau)^{-\frac 12(1-\frac 1p)} \| \wwhat{B}_2 \|_{L^q_{\xi_1}L^1_{\xi_2}} \|  \wwhat{\partial_2 v} \|_{L^2_{\xi_1}L^1_{\xi_2}} \\
		&\lesssim (t-\tau)^{-\frac 12(1-\frac 1p)} \| \wwhat{B}_2 \|_{L^1}^{\frac 2p}  \| \wwhat{B}_2 \|_{L^2_{\xi_1}L^1_{\xi_2}}^{1-\frac 2p}  \| \wwhat{\partial_2 v} \|_{L^2_{\xi_1}L^1_{\xi_2}} \\
		&\lesssim (t-\tau)^{-\frac 12(1-\frac 1p)} \| \wwhat{B}_2 \|_{L^1}^{\frac 2p}  \| B_2 \|_{L^2}^{\frac 12-\frac 1p} \| \partial_1 B \|_{L^2}^{\frac 12-\frac 1p} \| \wwhat{\partial_2 v} \|_{L^2_{\xi_1}L^1_{\xi_2}},
	\end{align*}
	where $1/p + 1 = 1/q + 1/2$ and $p \in (2,\infty)$. Then, with Young's convolution inequality, we deduce
	\begin{equation*}
	\begin{aligned}
		\left\| \int_{\Omega_3} \int_0^t \big| e^{-\xi_1^2(t-\tau)} \wwhat{\partial_1 (B_2\partial_2v)} (\xi,\tau) \big| \, \mathrm{d}\tau \mathrm{d}\xi \right\|_{L^{\frac{4}{3}}(0,T)} \lesssim G(T)^2.
	\end{aligned}
	\end{equation*}
	On the other hand, we have
	\begin{equation*}
		\begin{split}
			\left\| \int_{\Omega_3} \int_0^t \big| e^{-\xi_1^2(t-\tau)} \wwhat{\partial_1^2 (B \cdot \nabla)B} (\xi,\tau) \big| \, \mathrm{d}\tau \mathrm{d}\xi \right\|_{L^{\frac{4}{3}}(0,T)} \lesssim \|\wwhat{B_1 \partial_1 B}\|_{L_T^{\frac43} L^1} +  \|\wwhat{B_2 \partial_2 B}\|_{L_T^{\frac43} L^1} \lesssim  G(T)^2.
		\end{split}
	\end{equation*}
	Thus, it follows
	\begin{equation*}
		\mathcal{K}^{\pm}(\bbR^2) \lesssim G(T)^2.
	\end{equation*}
	Combining all the previous estimates, we finish the proof.
\end{proof}

\begin{proposition}\label{prop_B2}
	Let $m > 2$. Assume that $(v,B)\in C([0,\infty);H^m(\bbR^2))$ is a classical solution to \eqref{mhd_eq}. Then there holds
	\begin{equation*}
	\begin{gathered}
		\|\widehat{B}_2\|_{L_T^2 L^1} \lesssim E(0) + \int_{\bbR^2} \frac 1{|\xi|} \big| \widehat{B_0} \big| \,\mathrm{d}\xi + G(T)^2
		+ G(T)(\|\wwhat{\partial_1 v}\|_{L_T^1 L^1} + \| \widehat{v} \|_{L_T^{\frac 43}L^1}),
	\end{gathered}
	\end{equation*}
	for all $T \geq 0$.
\end{proposition}
\begin{proof}
    From \eqref{df_u} and the simple fact $\widehat{B}_2 = \langle  \sum_{\gamma=\pm} \widehat{\mathbf{u}},\mathbf{a}_{\gamma}^2 \rangle \langle \mathbf{b}_{\gamma}^2, e_4 \rangle$,
	we have
	\begin{equation}\label{eq_IJK}
		\left\| \int_{\bbR^2} \big|\langle \wwhat{ \mathbf{u}},\mathbf{a}_{\gamma}^2 \rangle \langle \mathbf{b}_{\gamma}^2, e_4 \rangle \big| \,\mathrm{d}\xi \right\|_{L^2(0,T)} \leq \mathcal{I}(\bbR^2) + \mathcal{J}^(\bbR^2) + \mathcal{K}(\bbR^2),
	\end{equation}
	where for any simply connected smooth domain $Q\subseteq \bbR^2$ we set 
\begin{equation*}
    \begin{split}
        \mathcal{I}^{+}(Q)&= \left\| \int_{Q} \big|e^{-\lambda_{+}(\xi)t}  \langle  \widehat{\mathbf{u}}_0(\xi),\mathbf{e_4}\rangle \big| \,\mathrm{d}\xi \right\|_{L^{2}(0,T)}, \\
        \mathcal{I}^{-}(Q)&=  \left\| \int_{Q} \big|(e^{-\lambda_{-}t}-e^{-\lambda_{+}t})  \langle  \widehat{\mathbf{u}}_0(\xi),\mathbf{a}_{-}^2(\xi) \rangle \langle \mathbf{b}_{-}^2, e_4 \rangle \big| \,\mathrm{d}\xi \right\|_{L^{2}(0,T)},
    \end{split}
\end{equation*}
in view of the anisotropic decomposition \eqref{decom_est_2}. 
The terms $\mathcal{J}^{\pm}(Q)$ and $\mathcal{K}^{\pm}(Q)$ are defined accordingly, see the proofs of Proposition~\ref{prop_vderiv} and Proposition~\ref{prop_Bderiv} for the inequality \eqref{eq_IJK}.

	We estimate the above quantities using \eqref{decom_est_2}. By the relation $|e^{-\lambda_{+}(\xi)t}| \leq e^{-\frac t2}$, we have that
	\begin{align*} 
 \mathcal{I}^{+}(\bbR^2) \leq \| e^{-\frac t2} \|_{L^2(0,T)} \int_{\bbR^2} \big|\widehat{B}_2(0) \big| \,\mathrm{d}\xi \lesssim  \| B_2(0) \|_{H^{m-1}}.
	\end{align*}
	By \eqref{Omg_1} and \eqref{Omg_2}, we can see
	\begin{equation*}
		\mathcal{I}^{-}(\Omega_1\cup \Omega_2) \lesssim  \left\| \int_{\bbR^2} e^{-\frac t2} \big| \widehat{\mathbf{u}}_0 \big| \,\mathrm{d}\xi \right\|_{L^2(0,T)} \lesssim  \| \mathbf{u}_0 \|_{H^{m-1}}.
	\end{equation*}
	From \eqref{Omg_3} combined with $
 \operatorname{div}B=0$, we can see for $\xi \in \Omega_3$ that
	\begin{equation}\label{B2_initial}
	\begin{aligned}
		\big|(e^{-\lambda_{-}t} - e^{-\lambda_{+}t}) \langle  \widehat{\mathbf{u}}_0,\mathbf{a}_{-}^2 \rangle \langle \mathbf{b}_{-}^2, e_4 \rangle \big| &\lesssim e^{-\xi_1^2t} \left( \big|\wwhat{ \partial_1 v}_2(0) \big| + \big| \wwhat{ B}_2(0) \big| \right) \\
		&\lesssim e^{-\xi_1^2t} \left( \frac {|\xi_1|^2}{|\xi|}\big| \wwhat{ v}_0 \big| + \frac {|\xi_1|}{|\xi|}\big| \wwhat{ B}_0 \big| \right).
	\end{aligned}
	\end{equation}
	This implies
	\begin{align*}
		\mathcal{I}^{-}(\Omega_3)\lesssim \bigg\| \int_{\bbR^2} \big|\xi_1 e^{-\xi_1^2t} \big| \big(\big|\widehat{v}_0 \big| + \frac 1{|\xi|} \big| \wwhat{ B}_0 \big| \big) \,\mathrm{d}\xi \bigg\|_{L^2(0,T)} \lesssim  \| v_0 \|_{H^{m-1}} + \int_{\bbR^2} \frac 1{|\xi|} \big| \widehat{B}_0 \big| \,\mathrm{d}\xi.
	\end{align*}
	Thus, we obtain $$\sum_{\gamma=\pm}\mathcal{I}^{\gamma}(\bbR^2) \lesssim \| \mathbf{u}_0 \|_{H^m} + \int_{\bbR^2} \frac 1{|\xi|} \big| \widehat{B}_0 \big| \,\mathrm{d}\xi.$$ We proceed with $\mathcal{J}^{\pm}(\bbR^2).$ The definition of $\mathbf{P}(\xi)$ and Young's convolution inequality give
	\begin{align*}
	\mathcal{J}^{+}(\bbR^2) &\lesssim \left\| \int_0^t e^{-\frac {t-\tau}2} \big| \wwhat{\partial_1 (v\otimes \mathbf{u})}(\xi,\tau) \big| \,\mathrm{d}\tau \right\|_{L_T^2 L^1} \lesssim \|\partial_1(v\otimes \mathbf{u})\|_{L_T^1 L^1}\lesssim G(T) \| \widehat{\partial_1 v}\|_{L_T^1 L^1} + G(T)^2,
	\end{align*}
 in light of the product rule $\partial_1(v\otimes u)= \partial_1 v\otimes u + v \partial_1 u.$
	Then we exploit \eqref{Omg_1} and \eqref{Omg_2} to obtain 
	\begin{align*}
	\mathcal{J}^{-}(\Omega_1\cup\Omega_2) &\lesssim \left\| \int_{\bbR^2}\int_0^t e^{-\frac {t-\tau}2} \big| \wwhat{\partial_1 (v\otimes \mathbf{u})}(\xi,\tau) \big| \, \mathrm{d}\tau\mathrm{d}\xi \right\|_{L^2(0,T)}  \lesssim  G(T)  \|\widehat{\partial_1 v}\|_{L_T^1 L^1} + G(T)^2.
	\end{align*}
	Using \eqref{Omg_3} and the definition of $\mathbf{P}(\xi)$ gives
	\begin{equation}\label{detail_est}
	\begin{split}
	\mathcal{J}^{-}(\Omega_3) &\lesssim\left\| \int_{\Omega_3}\int_0^t \big| (e^{-\lambda_{-}(\xi)(t-\tau)} - e^{-\lambda_{+}(\xi)(t-\tau)}) \wwhat{\partial_1(v\otimes \mathbf{u})}(\xi,\tau) \big| \, \mathrm{d}\tau\mathrm{d}\xi \right\|_{L^2(0,T)} \\
		&\lesssim\left\| \int_{\bbR^2}\int_0^t \big| \xi_1 e^{-\xi_1^2(t-\tau)} \wwhat{v\otimes \mathbf{u}}(\xi,\tau) \big| \, \mathrm{d}\tau\mathrm{d}\xi \right\|_{L^2(0,T)} \\
		&\lesssim\int_0^T\int_{\bbR^2} \big| \wwhat{v\otimes \mathbf{u}}(\xi,t) \big| \, \mathrm{d}\xi \mathrm{d}t \\
		&\lesssim \| \widehat{v} \|_{L_T^{\frac 43} L^1} \| \bfu \|_{L_T^4 L^1}.
	\end{split}
 \end{equation}
	Therefore we obtain
	\begin{equation*}
		\mathcal{J}^{\pm}(\bbR^2) \lesssim G(T)^2 + G(T) \int_0^T \int_{\bbR^2} \big| \widehat{\partial_1 v}(\xi,t) \big| \,\ud \xi \ud t + G(T)\| \widehat{v} \|_{L_T^{\frac 43} L^1}.
	\end{equation*}
	We clearly have
	\begin{equation}\label{ineq_Bu}
		\begin{split}\mathcal{K}^{+}(\bbR^2) &\leq \left\| \int_{\bbR^2}\int_0^t e^{-\frac {t-\tau}2} \big| \wwhat{\partial_1 (B \otimes \mathbf{u})}(\xi,\tau) \big| \, \mathrm{d}\tau\mathrm{d}\xi \right\|_{L^2(0,T)} \\
		& \lesssim\left\| \int_{\bbR^2}\big| \wwhat{\partial_1 (B\otimes \mathbf{u})} \big| \, \mathrm{d}\xi \right\|_{L^2(0,T)} \\
		& \lesssim\| \partial_1 \mathbf{u} \|_{L^2(0,T;H^{m-1})} \| \mathbf{u} \|_{L^{\infty}(0,T;H^m)} \\
		& \lesssim G(T)^2.
	\end{split}
 \end{equation}
	Utilizing \eqref{Omg_1} and \eqref{Omg_2}, we see that
\begin{equation}\label{B2_est_1}
		\big| \langle \mathbf{P}(\xi)\wwhat{(B\cdot \nabla)\widetilde{\mathbf{u}}}, \mathbf{a}_{-}^2 \rangle \langle \mathbf{b}_{-}^2,e_4 \rangle \big| \lesssim\big| \wwhat{\partial_1(B\otimes \bfu)} \big|, \qquad \xi \in\Omega_1 \cup \Omega_2.
	\end{equation}
	Then, we can obtain
	\begin{align*}
\mathcal{K}^{-}(\Omega_1\cup\Omega_2) \lesssim \left\| \int_{\bbR^2}\int_0^t e^{-\frac {t-\tau}2} \big| \wwhat{\partial_1 (B \otimes \mathbf{u})}(\xi,\tau) \big| \, \mathrm{d}\tau\mathrm{d}\xi \right\|_{L^2(0,T)}  \lesssim G(T)^2,
	\end{align*}
	similarly with the above estimate on $\mathcal{K}^{+}(\bbR^2)$. On the other hand, \eqref{Omg_3} gives
	\begin{equation}\label{B2_est_2}
		\big| \langle \mathbf{P}(\xi)\wwhat{(B\cdot \nabla)\widetilde{\mathbf{u}}}, \mathbf{a}_{-}^2 \rangle \langle \mathbf{b}_{-}^2,e_4 \rangle \big|\lesssim\big| \wwhat{\partial_1(B\otimes v)} \big| + C \big| \wwhat{\partial_1^2(B\otimes B)} \big|, \qquad \xi \in\Omega_3
	\end{equation}
	by using the fact $\xi_1 \leq \frac 14$ for any $\xi\in \Omega_3.$ So we have the bound $\mathcal{K}^{-}(\Omega_3)\lesssim \mathcal{K}_1^{-}(\Omega_3) + \mathcal{K}_2^{-}(\Omega_3)$ where the new quantities are defined in an obvious way according to \eqref{B2_est_2}. We can deduce as in \eqref{detail_est} that
	\begin{equation*}
	\begin{gathered}
	\mathcal{K}_{1}^{-}(\Omega_3) \lesssim \| \widehat{v} \|_{L_T^{\frac 43} L^1} \| B \|_{L_T^4 L^1}, \\
	\mathcal{K}_2^{-}(\Omega_3) \lesssim \left\| \int_{\bbR^2}\int_0^t \big| \xi_1^2 e^{-\lambda_{-}(\xi)(t-\tau)} \wwhat{B\otimes B}(\xi,\tau) \big| \, \mathrm{d}\tau\mathrm{d}\xi \right\|_{L^2(0,T)} \lesssim \bigg\| \int_{\bbR^2} \big| \wwhat{ B\otimes B} \big| \, \mathrm{d}\xi\bigg\|_{L^2(0,T)}  \lesssim G(T)^2.	
	\end{gathered}
	\end{equation*}
	Therefore, we establish
	\begin{equation*}
\sum_{\gamma=\pm}\mathcal{K}^{\gamma}(\bbR^2) \lesssim G(T)^2.
	\end{equation*}
	Collecting all the previous estimates, we complete the proof. 
\end{proof}

\begin{proposition}\label{prop_B2_1}
	Let $m \geq 4$. Assume that $(v,B)\in C([0,\infty);H^m(\bbR^2))$ is a classical solution to \eqref{mhd_eq}. Then there holds
	\begin{equation*}
	\begin{split}
		\|\widehat{\partial_1B}_2\|_{L_T^1 L^1} &\lesssim E(0) + \|\frac{\widehat{B}_0}{|\xi|}\|_{L^1} + G(T) (G(T)+\| \widehat{\partial_1v} \|_{L_T^1 L^1} +  \| \widehat{\partial_1B}_2 \|_{L_T^1 L^1} + \| \widehat{v} \|_{L_T^{\frac 43} L^1} + \| \widehat{\partial_1B} \|_{L_T^{\frac 43} L^1})
	\end{split}
	\end{equation*}
	for all $T \geq 0$.
\end{proposition}
\begin{proof}
	From \eqref{df_u} and $\widehat{\partial_1B}_2 = \sum_{\gamma=\pm} \langle \widehat{\partial_1 \mathbf{u}},\mathbf{a}_{\gamma}^2 \rangle \langle \mathbf{b}_{\gamma}^2, e_4 \rangle$, we have
	\begin{equation*}
		\left\| \int_{\bbR^2} \big|\sum_{\gamma=\pm} \langle \wwhat{ \partial_1\mathbf{u}},\mathbf{a}_{\gamma}^2 \rangle \langle \mathbf{b}_{\gamma}^2, e_4 \rangle \big| \,\mathrm{d}\xi \right\|_{L^1(0,T)} \leq \sum_{\gamma=\pm }\left(\mathcal{I}^{\gamma}(\bbR^2) + \mathcal{J}^{\gamma}(\bbR^2) + \mathcal{K}^{\gamma}(\bbR^2)\right),
	\end{equation*}
	where for any simply connected smooth domain $Q\subseteq \bbR^2$ we set, for example, 
\begin{equation*}
    \begin{split}
        \mathcal{I}^{+}(Q)&= \left\| \int_{Q} \big|e^{-\lambda_{+}(\xi)t}  \langle  \widehat{\partial_1 \mathbf{u}}_0(\xi),\mathbf{e_4}\rangle \big| \,\mathrm{d}\xi \right\|_{L^{2}(0,T)}, \\
        \mathcal{I}^{-}(Q)&=  \left\| \int_{Q} \big|(e^{-\lambda_{-}t}-e^{-\lambda_{+}t})  \langle  \widehat{\partial_1 \mathbf{u}}_0(\xi),\mathbf{a}_{-}^2(\xi) \rangle \langle \mathbf{b}_{-}^2, e_4 \rangle \big| \,\mathrm{d}\xi \right\|_{L^{2}(0,T)}.
    \end{split}
\end{equation*}
Note that the quantities $\mathcal{J}^{\pm}$ and $\mathcal{K}^{\pm}$ are defined in the obvious fashion so that they correspond to the nonlinear parts
$ \mathbf{P}(\xi)\wwhat{\partial_1(v\cdot\nabla)\mathbf{u}}$ and $ \mathbf{P}(\xi)\wwhat{\partial_1 (B\cdot\nabla)\mathbf{u}}$, respectively, in view of the anisotropic decomposition \eqref{decom_est_2}. For the details, see the proofs of the previous propositions.

	We can estimate $\mathcal{I}^{\pm}$ similarly with the proof of Proposition~\ref{prop_B2}. See that
	\begin{align*}
	\mathcal{I}^{+}(\bbR^2) &\leq \| e^{-\frac t2} \|_{L^1(0,T)} \|\widehat{\partial_1 B}_2(0)\|_{L^1} \lesssim \| B_2(0) \|_{H^{m-1}}, \\
 \mathcal{I}^{-}(\Omega_1\cup\Omega_2) &\lesssim \left\| \int_{\bbR^2} e^{-\frac t2} \big| \widehat{\partial_1 \mathbf{u}}_0 \big| \,\mathrm{d}\xi \right\|_{L^1(0,T)} \lesssim \| \mathbf{u}_0 \|_{H^{m-1}},\\
 \mathcal{I}^{-}(\Omega_3) &\lesssim \bigg\| \int_{\bbR^2} \big|\xi_1^2 e^{-\xi_1^2t} \big| \big(\big|\widehat{v}_0 \big| + \frac {| \wwhat{ B}_0 |}{|\xi|}  \big) \,\mathrm{d}\xi \bigg\|_{L^1(0,T)} \lesssim \| v_0 \|_{H^{m-2}} + \int_{\bbR^2} \frac {| \widehat{B}_0 |}{|\xi|}  \,\mathrm{d}\xi,
	\end{align*}
where we used \eqref{B2_initial} for the estimate of $\mathcal{I}^{-}(\Omega_3).$
	This leads to
	$\mathcal{I}^{\pm}(\bbR^2) \lesssim \| \mathbf{u}_0 \|_{H^m} + \int_{\bbR^2} \frac {| \widehat{B}_0 |}{|\xi|}  \,\mathrm{d}\xi.$
	Leveraging the definition of $\mathbf{P}(\xi)$ and Young's convolution inequality, we get
	\begin{align*}
		\mathcal{J}^{+}(\bbR^2) &\leq \left\| \int_{\bbR^2}\int_0^t e^{-\frac {t-\tau}2} \big| \wwhat{\partial_1^2 (v\otimes \mathbf{u})}(\xi,\tau) \big| \, \mathrm{d}\tau\mathrm{d}\xi \right\|_{L^1(0,T)} \\
		 &\lesssim \| \wwhat{ \partial_1^2v\otimes \mathbf{u}}\|_{L_T^1 L^1} + \| \wwhat{ \partial_1v\otimes \partial_1 \mathbf{u}}\|_{L_T^1 L^1} + \| \wwhat{v\otimes \partial_1^2\mathbf{u}}\|_{L_T^1 L^1} \\ &\lesssim  G(T) \| \widehat{\partial_1 v}\|_{L_T^1 L^1} + G(T)^2.
	\end{align*}
	In the above computation, we employed the interpolation inequality $\int_{\bbR^2} \big| \widehat{\partial_1^2v} \big| \,\ud \xi \lesssim \| \widehat{\partial_1v} \|_{L^1}^{\frac 12} \| v \|_{H^4}^{\frac 12}$. Meanwhile, \eqref{Omg_1} and \eqref{Omg_2} yield the estimate for $\mathcal{J}^{-}(\Omega_1\cup\Omega_2)$ as
	\begin{align*}
		\mathcal{J}^{-}(\Omega_1\cup\Omega_2) &\leq \left\| \int_{\bbR^2}\int_0^t e^{-\frac {t-\tau}2} \big| \wwhat{\partial_1^2 (v\otimes \mathbf{u})}(\xi,\tau) \big| \, \mathrm{d}\tau\mathrm{d}\xi \right\|_{L^1(0,T)} \lesssim G(T) \| \widehat{\partial_1 v}\|_{L_T^1 L^1} + G(T)^2.
	\end{align*}
	To bound $\mathcal{J}^{-}(\Omega_3)$, we exploit \eqref{Omg_3} and the definition of $\mathbf{P}(\xi)$ to obtain
\begin{equation*}
	\begin{aligned}
		\mathcal{J}^{-}(\Omega_3)
		\lesssim \left\| \int_{\bbR^2}\int_0^t \big| \xi_1^2 e^{-\xi_1^2(t-\tau)} \wwhat{v\otimes \mathbf{u}}(\xi,\tau) \big| \, \mathrm{d}\tau\mathrm{d}\xi \right\|_{L^1(0,T)} \lesssim  \|\wwhat{v\otimes \mathbf{u}}\|_{L_T^1 L^1} \lesssim \| \widehat{v} \|_{L_T^{\frac 43} L^1} \| \bfu \|_{L_T^4 L^1}.
	\end{aligned}
	\end{equation*}
It follows that	\begin{equation*}
		\sum_{\gamma=\pm}\mathcal{J}^{\gamma}(\bbR^2) \lesssim G(T)^2 + G(T) \|\wwhat{\partial_1 v}\|_{L_T^1 L^1} + G(T)\| \widehat{v} \|_{L_T^{\frac 43} L^1}.
	\end{equation*}
It suffices to estimate $\mathcal{K}^{\pm}(\bbR^2).$ We begin with the observation that
	\begin{align*}
		&\left\| \int_{\bbR^2}\int_0^t \big| e^{-\lambda_{+}(\xi)(t-\tau)} \langle \mathbf{P}(\xi)\wwhat{\partial_1 (B\cdot \nabla)\widetilde{\mathbf{u}}}(\xi,\tau),e_4 \rangle \big| \, \mathrm{d}\tau\mathrm{d}\xi \right\|_{L^1(0,T)} \\
		&\hphantom{\qquad} \leq \left\| \int_{\bbR^2}\int_0^t e^{-\frac {t-\tau}2} \big| \wwhat{\partial_1^2 (B \otimes v)}(\xi,\tau) \big| \, \mathrm{d}\tau\mathrm{d}\xi \right\|_{L^1(0,T)} + \left\| \int_{\bbR^2}\int_0^t e^{-\frac {t-\tau}2} \big| \wwhat{\partial_1^2 (B \otimes B)}(\xi,\tau) \big| \, \mathrm{d}\tau\mathrm{d}\xi \right\|_{L^1(0,T)}.
	\end{align*}
	In the above inequality, we note that the first term has been already estimated in \eqref{ineq_Bu} during the proof of Proposition~\ref{prop_B2}, and it is bounded below by $G(T)^2$. Then we need to control the second quantity in the above. The two simple interpolation inequalities 
	$\| \widehat{\partial_1^2B_1} \|_{L^1} \lesssim \| \widehat{\nabla \partial_1 B_1} \|_{L^1} \lesssim \| \widehat{\partial_1B}_2 \|_{L^1}^{\frac 12} \| \partial_1 B \|_{H^3}^{\frac 12}$ and
$\|\widehat{\partial_1^2B}_2\|_{L^1} \lesssim \| \widehat{\partial_1B}_2 \|_{L^1}^{\frac 12} \| \partial_1 B \|_{H^3}^{\frac 12}$ allows us to bound the second term as
\begin{equation*}
		\left\| \int_{\bbR^2}\int_0^t e^{-\frac {t-\tau}2} \big| \wwhat{\partial_1^2 (B \otimes B)}(\xi,\tau) \big| \, \mathrm{d}\tau\mathrm{d}\xi \right\|_{L^1(0,T)} \lesssim G(T) \| \widehat{\partial_1B}_2 \|_{L_T^1 L^1} + G(T)^2,
	\end{equation*}
which means that we have
	\begin{equation*}
 \mathcal{K}^{+}(\bbR^2) \lesssim G(T) \| \widehat{\partial_1v} \|_{L_T^1 L^1} + G(T) \| \widehat{\partial_1B}_2 \|_{L_T^1 L^1} + G(T)^2.
\end{equation*}
	Using \eqref{B2_est_1}, we also have
	\begin{align*}
\mathcal{K}^{-}(\Omega_1\cup\Omega_2) \leq \left\| \int_0^t e^{-\frac {t-\tau}2} \big| \wwhat{\partial_1 (B \otimes \mathbf{u})}(\xi,\tau) \big| \, \mathrm{d}\tau \right\|_{L_T^1 L^1} \lesssim G(T) (\| \widehat{\partial_1v} \|_{L_T^1 L^1} +  \| \widehat{\partial_1B}_2 \|_{L_T^1 L^1}) + G(T)^2.
	\end{align*}
	On the other hand, \eqref{Omg_3} gives \begin{equation}\label{B2deriv_est}
		\big| \langle \mathbf{P}(\xi)\wwhat{(B\cdot \nabla)\widetilde{\partial_1 \mathbf{u}}}, \mathbf{a}_{-}^2 \rangle \langle \mathbf{b}_{-}^2,e_4 \rangle \big|\lesssim\big| \wwhat{\partial_1^2(B\otimes v)} \big| + C \big| \wwhat{\partial_1^3(B\otimes B)} \big|, \qquad \xi \in\Omega_3
	\end{equation}
	by using the fact $\xi_1 \leq \frac 14$ for any $\xi\in \Omega_3.$ Thus we have the bound $\mathcal{K}^{-}(\Omega_3)\lesssim \mathcal{K}_1^{-}(\Omega_3) + \mathcal{K}_2^{-}(\Omega_3)$ where the new quantities are defined in an obvious way according to \eqref{B2deriv_est}. We see that \begin{equation*}
		\mathcal{K}_1^{-}(\Omega_3) \leq \left\| \int_{\Omega_3}\int_0^t \big| \xi_1^2e^{-\xi_1^2(t-\tau)} \wwhat{B\otimes v}(\xi,\tau) \big| \, \mathrm{d}\tau\mathrm{d}\xi \right\|_{L^1(0,T)} \lesssim G(T)\| \widehat{v} \|_{L_T^{\frac 43} L^1}.
	\end{equation*}
	We further compute
	\begin{align*}
		\mathcal{K}_2^{-}(\Omega_3) \lesssim \left\| \int_0^t \big| \xi_1^2 e^{-\xi_1^2(t-\tau)} \wwhat{\partial_1(B\otimes B)}(\xi,\tau) \big| \, \mathrm{d}\tau \right\|_{L_T^1 L^1} \lesssim \|\wwhat{ \partial_1B\otimes B} \|_{L_T^1 L^1} \lesssim G(T) \| \widehat{\partial_1B} \|_{L_T^{\frac 43} L^1}.	
	\end{align*}
	This concludes that
	\begin{equation*}
\sum_{\gamma=\pm}\mathcal{K}^{\gamma}(\bbR^2) \lesssim G(T)^2 + G(T) (\| \widehat{\partial_1v} \|_{L_T^1 L^1} +  \| \widehat{\partial_1B}_2 \|_{L_T^1 L^1} + \| \widehat{v} \|_{L_T^{\frac 43} L^1} + \| \widehat{\partial_1B} \|_{L_T^{\frac 43} L^1}).	\end{equation*}
	The proof is finished.
\end{proof}

\subsection{Proof of Theorem~\ref{thm1}}
\begin{proof}[Proof of Theorem~\ref{thm1}]
 Assume that $m \geq 4$ and $\| (v_0,B_0) \|_{X^m} \leq \delta$ for a sufficiently small $\delta>0$. Let $(v,B)$ be the local-in-time solution obtained in Proposition~\ref{loc_prop} such that $(v,B) \in C([0,T_{\ast});H^m(\bbR^2))$ for the maximal time $T_{\ast} > 0$. To conclude that $T_{\ast} = \infty$, it suffices to show that $\sup_{T \in [0,T_{\ast})} G(T) $ is finite. For simplicity, we define a mediating quantity $H(T)$ as 
\begin{gather*}
    H(T) := \|\widehat{\partial_1B}_2\|_{L_T^1 L^1} + \|\widehat{B}_2\|_{L_T^2 L^1} + \|\widehat{\nabla B}_2\|_{L_T^{\frac43}L^1} + \|\widehat{\partial_1 v}\|_{L_T^1 L^1} + \| \frac {|\widehat{v}_2|}{\sqrt{|\xi_1|}} \|_{L^2(0,T;L^2_{\xi_1}L^1_{\xi_2})} + \| \frac {|\widehat{v}_2|}{\sqrt{|\xi_1|}} \|_{L_T^{\frac 43} L^1}.
\end{gather*}
From Proposition~\ref{prop_v}--\ref{prop_B2_1}, we deduce the following. There exists a constant $C_1>0$ such that $$H(T) \leq C_1\delta + C_1G(T)^2 + C_1G(T)H(T), \qquad T \in [0,T_{\ast}).$$ On the other hand, we have from \eqref{Em_ineq} and \eqref{A_est} that
\begin{gather*}
    G(T)^2 \leq 3E(0)^2 + C_2G(T)^3 + C_2G(T)^2 (\|\widehat{\partial_1 v}\|_{L_T^1 L^1} + \|\widehat{B}_2\|_{L_T^2 L^1}), \qquad T \in [0,T_{\ast})
\end{gather*}
for some $C_2>0$. Since $G(\cdot)$ is continuous, there exists $T \in [0,T_{\ast})$ with $G(T) \leq 4E(0)$. Then for $\delta>0$ with $4C_1\delta \leq \frac 12$, there holds $H(T) \leq 2C_1\delta + 2C_1G(T)^2$. Therefore, it follows that $$G(T)^2 \leq 3E(0)^2 + 2C_1 C_2 \delta G(T)^2 + C_2G(T)^3 + 2C_1C_2 G(T)^4.$$ By taking sufficiently small $\delta>0$, we can guarantee that $G(T)^2 \leq 4E(0)^2$. Repeating the above process, we reach that $\sup_{T \in [0,T_{\ast})} G(T)^2 \leq 4E(0)^2$. This completes the proof.
\end{proof}

\section{Appendix}
\subsection{Proof of Proposition~\ref{prop_v}}
	
	We bound 
		$\| \frac {|\widehat{v}_2|}{\sqrt{|\xi_1|}} \|_{L_T^{\frac 43} L^1}$ first. From \eqref{df_u} and the simple fact $\widehat{v}_2 = \sum_{\gamma=\pm} \langle  \widehat{\mathbf{u}},\mathbf{a}_{\gamma}^2 \rangle \langle \mathbf{b}_{\gamma}^2, e_2 \rangle$,
	we have
	\begin{equation*}
		\left\| \int_{\bbR^2} \frac 1{\sqrt{|\xi_1|}} \big|\sum_{\gamma=\pm} \langle \wwhat{ \mathbf{u}},\mathbf{a}_{\gamma}^2 \rangle \langle \mathbf{b}_{\gamma}^2, e_2 \rangle \big| \,\mathrm{d}\xi \right\|_{L^{\frac43}(0,T)} \leq \sum_{\gamma=\pm}\left(\mathcal{I}^{\gamma}(\bbR^2) + \mathcal{J}^{\gamma}(\bbR^2) + \mathcal{K}^{\gamma}(\bbR^2)\right),
	\end{equation*}
	where for any simply connected smooth domain $Q\subseteq \bbR^2$ we set, for example,
\begin{align*}
    \mathcal{I}^{+}(Q)&:=\left\| \int_{\bbR^2} \frac 1{\sqrt{|\xi_1|}} \big|e^{-\lambda_{+}t} \langle \widehat{\mathbf{u}}_0, e_2 \rangle \big| \,\mathrm{d}\xi \right\|_{L^{\frac 43}(0,T)} \\
    \mathcal{I}^{-}(Q)&:=\left\| \int_{Q} \frac 1{\sqrt{|\xi_1|}} \big|(e^{-\lambda_{-}t}-e^{-\lambda_{+}t}) \langle  \widehat{\mathbf{u}}_0(\xi),\mathbf{a}_{-}^2(\xi) \rangle \langle \mathbf{b}_{-}^2, e_2 \rangle \big| \,\mathrm{d}\xi \right\|_{L^{\frac43}(0,T)},
\end{align*}
in view of the anisotropic decomposition \eqref{decom_est_1}. Recall the decomposition
\begin{equation*}\label{decomp}
    \sum_{\gamma=\pm} e^{-\lambda_{\gamma}t} \langle  \widehat{\mathbf{f}},\mathbf{a}_{\gamma}^2 \rangle \langle \mathbf{b}_{\gamma}^2, e_2 \rangle = (e^{-\lambda_{-}t} - e^{-\lambda_{+}t}) \langle  \widehat{\mathbf{f}},\mathbf{a}_{-}^2 \rangle \langle \mathbf{b}_{-}^2, e_2 \rangle + e^{-\lambda_{+}t} \langle \widehat{\mathbf{f}}, e_2 \rangle ,\quad \forall \widehat{\mathbf{f}} \in \bbC^4.
\end{equation*}
 The quantities $\mathcal{J}^{\pm}(Q)$ and $\mathcal{K}^{\pm}(Q)$ are also defined accordingly to such a decomposition, which correspond to the nonlinear terms $\mathbf{P}(\xi)\wwhat{(v\cdot \nabla)\mathbf{u}}(\xi)$ and $\mathbf{P}(\xi)  \wwhat{(B\cdot \nabla)\widetilde{\mathbf{u}}}(\xi)$ that stem from the Duhamel formula \eqref{df_u}, respectively.
 We begin by bounding $\mathcal{I}^{\pm}(\bbR^2).$ By the relation $|e^{-\lambda_{+}(\xi)t}| \leq e^{-\frac t2}$, we have that
	\begin{align*}
	\mathcal{I}^{+}(\bbR^2) \leq \| e^{-\frac t2} \|_{L^{\frac 43}(0,T)} \int_{\bbR^2} \frac 1{\sqrt{|\xi_1|}}\big|\widehat{v}_2(0) \big| \,\mathrm{d}\xi \lesssim  \int_{\bbR^2} \frac {\sqrt{|\xi_1|}}{|\xi|}\big|\widehat{v}_0 \big| \,\mathrm{d}\xi \lesssim \| v_0 \|_{H^1}.
	\end{align*}
 Note that we used $\operatorname{div}v=0$ in the above computation.
	By \eqref{Omg_1}, \eqref{Omg_2}, and the simple fact $\frac{1}{\sqrt{|\xi_1|}} \lesssim 1 $ for $\xi\in \Omega_1\cup\Omega_2$, we see that
	\begin{align*}
		\mathcal{I}^{-}(\Omega_1\cup\Omega_2) \lesssim  \left\| \int_{\bbR^2} e^{-\frac t2} \big| \widehat{\mathbf{u}}_0 \big| \,\mathrm{d}\xi \right\|_{L^{\frac 43}(0,T)} \lesssim \| \mathbf{u}_0 \|_{H^{2}}.
	\end{align*}
	From \eqref{Omg_3}  we can see for $\xi \in \Omega_3$ that
	\begin{equation}\label{v_initial}
	\begin{aligned}
		\frac 1{\sqrt{|\xi_1|}} \big|(e^{-\lambda_{-}t} - e^{-\lambda_{+}t}) \langle  \widehat{\mathbf{u}}_0,\mathbf{a}_{-}^2 \rangle \langle \mathbf{b}_{-}^2, e_2 \rangle \big| &\leq Ce^{-\xi_1^2t} \frac 1{\sqrt{|\xi_1|}} \left( \big|\wwhat{ \partial_1^2 v}_2(0) \big| + \big| \wwhat{ \partial_1B}_2(0) \big| \right) \\
		&\leq Ce^{-\xi_1^2t} \left( \frac {|\xi_1|^{\frac 52}}{|\xi|}\big| \wwhat{ v}_0 \big| + \frac {|\xi_1|^{\frac 32}}{|\xi|}\big| \wwhat{ B}_0 \big| \right).
	\end{aligned}
	\end{equation}
	This implies
	\begin{align*}
		\mathcal{I}^{-}(\Omega_3) \lesssim \bigg\| \int_{\bbR^2} \big|\xi_1^{\frac 32} e^{-\xi_1^2t} \big| \big(\big|\widehat{v}_0 \big| + \frac 1{|\xi|} \big| \wwhat{ B}_0 \big| \big) \,\mathrm{d}\xi \bigg\|_{L^{\frac 43}(0,T)} \lesssim \| v_0 \|_{H^{1+}} + \int_{\bbR^2} \frac 1{|\xi|} \big| \widehat{B}_0 \big| \,\mathrm{d}\xi.
	\end{align*}
Therefore, we obtain	
		$\sum_{\gamma}\mathcal{I}^{\gamma}(\bbR^2) \lesssim \| \mathbf{u}_0 \|_{H^m} + \||\xi|^{-1}\widehat{B}_0\|_{L^1}.$ Now we estimate $\mathcal{J}^{\pm}(\bbR^2).$ With the definition of $\mathbf{P}(\xi)$ and Young's convolution inequality, we have
	\begin{align*}
		 \mathcal{J}^{+}(\bbR^2) &\leq \left\| \int_{\bbR^2}\int_0^t e^{-\frac {t-\tau}2} \sqrt{|\xi_1|} \big| \wwhat{v\otimes \mathbf{u}}(\xi,\tau) \big| \, \mathrm{d}\tau\mathrm{d}\xi \right\|_{L^{\frac 43}(0,T)} \\
		&\lesssim \int_0^T\int_{\bbR^2} \big| \sqrt{|\xi_1|} \widehat{v} \big| * \big| \widehat{\mathbf{u}} \big| (\xi,t) \, \mathrm{d}\xi \mathrm{d}t + \int_0^T\int_{\bbR^2} \big| \widehat{v} \big| * \big| \sqrt{|\xi_1|} \widehat{\mathbf{u}} \big| (\xi,t) \, \mathrm{d}\xi \mathrm{d}t \\
		&\lesssim G(T) \int_0^T \int_{\bbR^2} \big| \widehat{\partial_1 v}(\xi,t) \big| \,\ud \xi \ud t +  G(T) \| \widehat{v} \|_{L_T^{\frac 43} L^1} + G(T)^2.
	\end{align*}
	Using \eqref{Omg_1} and \eqref{Omg_2}, we obtain 
	\begin{align*}
		\mathcal{J}^{-}(\Omega_1\cup \Omega_2)\lesssim \left\| \int_0^t e^{-\frac {t-\tau}2} \sqrt{|\xi_1|} \big| \wwhat{v\otimes \mathbf{u}}(\xi,\tau) \big| \, \mathrm{d}\tau \right\|_{L_T^{\frac 43} L^1} \lesssim G(T)( \|\widehat{\partial_1 v}\|_{L_T^1 L^1} +  \| \widehat{v} \|_{L_T^{\frac 43} L^1}) + G(T)^2.
	\end{align*}
	Furthermore, we leverage \eqref{Omg_3} and the definition of $\mathbf{P}(\xi)$ to yield
	\begin{equation}\label{detail_est2}
	\begin{aligned}
		\mathcal{J}^{-}(\Omega_3)&\lesssim \left\| \int_{\bbR^2}\int_0^t \big| \xi_1^{\frac 32} e^{-\xi_1^2(t-\tau)} \wwhat{v\otimes \mathbf{u}}(\xi,\tau) \big| \, \mathrm{d}\tau\mathrm{d}\xi \right\|_{L^{\frac 43}(0,T)} \\
		&\lesssim \int_0^T\int_{\bbR^2} \big| \wwhat{v\otimes \mathbf{u}}(\xi,t) \big| \, \mathrm{d}\xi \mathrm{d}t \\
		&\lesssim \| \widehat{v} \|_{L_T^{\frac 43} L^1} \| \widehat{\bfu} \|_{L_T^4 L^1}.
	\end{aligned}
	\end{equation}
	It suffices to bound $\mathcal{K}^{\pm}(\bbR^2)$. Due to Young's convolution inequality, we compute
	\begin{equation*}
	\mathcal{K}^{+}(\bbR^2) \lesssim \left\| \int_0^t e^{-\frac {t-\tau}2} \sqrt{|\xi_1|} \big| \wwhat{B \otimes \mathbf{u}}(\xi,\tau) \big| \, \mathrm{d}\tau \right\|_{L_T^{\frac 43} L^1} \lesssim \left\| \sqrt{|\xi_1|} \big| \wwhat{B\otimes \mathbf{u}} \big|\right\|_{L_T^{\frac 43} L^1} \lesssim G(T) \| \widehat{\partial_1 B} \|_{L_T^{\frac 43} L^1}
 \end{equation*}
	From \eqref{Omg_1} and \eqref{Omg_2} we deduce that
	\begin{equation}
	\mathcal{K}^{-}(\Omega_1\cup \Omega_2) \lesssim \left\| \int_{\bbR^2}\int_0^t e^{-\frac {t-\tau}2} \big| \sqrt{|\xi_1|} \wwhat{B \otimes \mathbf{u}}(\xi,\tau) \big| \, \mathrm{d}\tau\mathrm{d}\xi \right\|_{L^{\frac 43}(0,T)} \lesssim G(T)\| \widehat{\partial_1 B} \|_{L_T^{\frac 43} L^1}.
	\end{equation}
	On the other hand, \eqref{Omg_3} and the definition of $\bfP$ gives
	\begin{equation*}
		\big| \langle \mathbf{P}(\xi)\wwhat{(B\cdot \nabla)\widetilde{\mathbf{u}}}, \mathbf{a}_{-}^2 \rangle \langle \mathbf{b}_{-}^2,e_2 \rangle \big|\leq C \big| \wwhat{\partial_1^2(B\otimes v)} \big| + C \big| \wwhat{\partial_1^3(B\otimes B)} \big|, \qquad \xi \in\Omega_3.
	\end{equation*}
	This implies that 
	\begin{equation*}
	\mathcal{K}^{-}(\Omega_3) \lesssim 
		\left\| \int_0^t \big| \xi_1^{\frac 32} e^{-\xi_1^2(t-\tau)} \wwhat{B\otimes v}(\xi,\tau) \big| \, \mathrm{d}\tau \right\|_{L_T^{\frac 43} L^{1}(\Omega_3)} + \left\| \int_0^t \big| \xi_1^{\frac 32} e^{-\xi_1^2(t-\tau)} \wwhat{\partial_1B\otimes B}(\xi,\tau) \big| \, \mathrm{d}\tau \right\|_{L_T^{\frac 43} L^{1}(\Omega_3)}.
	\end{equation*}
The two terms on the right-hand side of the above inequality can be bounded by
 $\| \widehat{v} \|_{L_T^{\frac 43} L^1} \| \widehat{B} \|_{L_T^4 L^1}$ and $\| \widehat{\partial_1 B} \|_{L_T^{\frac 43} L^1} \| \widehat{B} \|_{L_T^4 L^1}$, respectively. This concludes that $\sum_{\gamma=\pm}\mathcal{K}^{\gamma}(\bbR^2) \lesssim G(T) \| \widehat{v} \|_{L_T^{\frac 43} L^1} + G(T) \| \widehat{\partial_1 B} \|_{L_T^{\frac 43} L^1}.$ Collecting the previous estimates, we complete the proof of the bound for $\| \frac {|\widehat{v}_2|}{\sqrt{|\xi_1|}} \|_{L_T^{\frac 43} L^1}$ in Proposition~\ref{prop_v}.

	The next claim we prove here is that $\| \frac {|\widehat{v}_2|}{\sqrt{|\xi_1|}} \|_{L^2(0,T;L^2_{\xi_1}L^1_{\xi_2})} \lesssim \| \bfu_0 \|_{H^m} + \| \frac {|\widehat{\bfu}_0|}{\sqrt{|\xi|}} \|_{L^2_{\xi_1}L^1_{\xi_2}} + G(T)^2.$ From \eqref{df_u} and the simple fact $\widehat{v}_2 = \sum_{\gamma=\pm} \langle  \widehat{\mathbf{u}},\mathbf{a}_{\gamma}^2 \rangle \langle \mathbf{b}_{\gamma}^2, e_2 \rangle$,
	we have
	\begin{equation*}
		\left\| \frac 1{\sqrt{|\xi_1|}} \big|\sum_{\gamma=\pm}\langle \wwhat{ \mathbf{u}},\mathbf{a}_{\gamma}^2 \rangle \langle \mathbf{b}_{\gamma}^2, e_2 \rangle \big|\right\|_{L^2(0,T;L^2_{\xi_1}L^1_{\xi_2})} \leq \sum_{\gamma=\pm}\left(\mathcal{I}^{\gamma}(\bbR^2) + \mathcal{J}^{\gamma}(\bbR^2) + \mathcal{K}^{\gamma}(\bbR^2)\right),
	\end{equation*}
	where for any simply connected smooth domain $Q\subseteq \bbR^2$ we set
 \begin{equation*}
 \begin{split}
     \mathcal{I}^{+}(Q)&:=\left\| \frac 1{\sqrt{|\xi_1|}} \big|e^{-\lambda_{+}t} \langle \widehat{\mathbf{u}}_0, e_2 \rangle \big| \right\|_{L^2(0,T;L^2_{\xi_1}L^1_{\xi_2})}, \\
     \mathcal{I}^{-}(Q)&:=\left\| \int_{Q} \frac 1{\sqrt{|\xi_1|}} \big|(e^{-\lambda_{-}t}-e^{-\lambda_{+}t})  \langle  \widehat{\mathbf{u}}_0(\xi),\mathbf{a}_{-}^2(\xi) \rangle \langle \mathbf{b}_{-}^2, e_2 \rangle \big| \,\mathrm{d}\xi \right\|_{L^2(0,T;L^2_{\xi_1}L^1_{\xi_2})}
     \end{split},
 \end{equation*}
 almost identically to the previous setting for the estimate of $\| \frac {|\widehat{v}_2|}{\sqrt{|\xi_1|}} \|_{L_T^{\frac 43} L^1}$, except that here we measure with respect to the norm of $L_T^2 L^2_{\xi_1}L^1_{\xi_2}$ instead of the norm of $L_T^{\frac 43} L^1$. The other terms, $\mathcal{J}^{\pm}(Q)$ and $\mathcal{K}^{\pm}(Q)$ are defined correspondingly to the nonlinearities that appear in \eqref{df_u}, in view of the anisotropic decomposition \eqref{decom_est_1}, as we did repeatedly throughout the paper. By the relation $|e^{-\lambda_{+}(\xi)t}| \leq e^{-\frac t2}$ and the incompressibility 
condition $\operatorname{div}v=0$, we have that
	\begin{align*}
	\mathcal{I}^{+}(\bbR^2) \lesssim \| e^{-\frac t2} \|_{L^2(0,T)} \left\| \frac 1{\sqrt{|\xi_1|}}\big|\widehat{v}_2(0) \big| \right\|_{L^2_{\xi_1}L^1_{\xi_2}} \lesssim \left\| \frac {\sqrt{|\xi_1|}}{|\xi|}\big|\widehat{v}_0 \big| \right\|_{L^2_{\xi_1}L^1_{\xi_2}}.
	\end{align*}
	By \eqref{Omg_1} and \eqref{Omg_2}, we see that
	\begin{align*}
	\mathcal{I}^{-}(\Omega_1\cup\Omega_2) \lesssim \left\| (e^{-\lambda_{-}t} - e^{-\lambda_{+}t}) \langle  \widehat{\mathbf{u}}_0,\mathbf{a}_{-}^2 \rangle \langle \mathbf{b}_{-}^2, e_2 \rangle  \right\|_{L_T^2 L^2_{\xi_1}L^1_{\xi_2}(\Omega_1 \cup \Omega_2)} \lesssim \left\| e^{-\frac t2} \big| \widehat{\bfu}_0 \big| \right\|_{L_T^2 L^2_{\xi_1}L^1_{\xi_2}}\lesssim \left\| \bfu_0 \right\|_{H^m}.
	\end{align*}
	By \eqref{v_initial} we also obtain
	\begin{align*}
		\mathcal{I}^{-}(\Omega_3) \lesssim \bigg\| \big|\xi_1^{\frac 32} e^{-\xi_1^2t} \big| \big(\big|\widehat{v}_0 \big| + \frac 1{|\xi|} \big| \wwhat{ B}_0 \big| \big) \bigg\|_{L_T^2 L^2_{\xi_1}L^1_{\xi_2}} \lesssim \| v_0 \|_{H^m} + \left\| \frac {\sqrt{|\xi_1|}}{|\xi|}\big|\widehat{B}_0 \big| \right\|_{L^2_{\xi_1}L^1_{\xi_2}}.	\end{align*}
This gives the desired bound for $\mathcal{I}^{\pm}(\bbR^23).$ To bound $\mathcal{J}^{\pm}(\bbR^2)$, we start by using the definition of $\mathbf{P}(\xi)$ and Young's convolution inequality so that we get
	\begin{align*}
\mathcal{J}^{+}(\bbR^2) &\lesssim \left\| \int_0^t e^{-\frac {t-\tau}2} \sqrt{|\xi_1|} \big\| \wwhat{v\otimes \mathbf{u}}(\xi,\tau) \big\|_{L^1_{\xi_2}} \, \mathrm{d}\tau \right\|_{L^2(0,T;L^2_{\xi_1})} \\
		&\lesssim \big\| \sqrt{|\xi_1|} \big|\widehat{v}\big| * \big| \widehat{\mathbf{u}} \big| \big\|_{L^2(0,T;L^2_{\xi_1}L^1_{\xi_2})} +  \big\| \big| {v} \big| * \sqrt{|\xi_1|} \big| \widehat{\mathbf{u}} \big| \big\|_{L^2(0,T;L^2_{\xi_1}L^1_{\xi_2})} \\
		&\lesssim G(T)^2.
	\end{align*}
Furthermore, relying on \eqref{Omg_1} and \eqref{Omg_2}, we estimate as 
	\begin{equation*}
		\mathcal{J}^{-}(\Omega_2\cup\Omega_3) \lesssim \left\| \int_0^t e^{-\frac {t-\tau}2} \sqrt{|\xi_1|} \big\| \wwhat{v\otimes \mathbf{u}}(\xi,\tau) \big\|_{L^1_{\xi_2}} \, \mathrm{d}\tau \right\|_{L_T^2 L^2_{\xi_1}} \lesssim G(T)^2
	\end{equation*}
Using \eqref{Omg_3} and the definition of $\mathbf{P}(\xi)$ gives
	\begin{equation}\label{detail_est3}
\mathcal{J}^{-}(\Omega_3) \lesssim \left\| \int_0^t \xi_1^{\frac 32} e^{-\xi_1^2(t-\tau)} \big\| \wwhat{v\otimes \mathbf{u}}(\xi,\tau) \big\|_{L^1_{\xi_2}} \, \mathrm{d}\tau \right\|_{L_T^2 L^2_{\xi_1}} \lesssim \big\| \wwhat{v\otimes \mathbf{u}} \big\|_{L_T^{\frac 43} L^2_{\xi_1}L^1_{\xi_2}} \lesssim \| \widehat{v} \|_{L_T^2 H^m} \| \widehat{B} \|_{L_T^4 L^1},
	\end{equation}
 and thus $\mathcal{J}^{-}(\Omega_3) \lesssim G(T)^2$. Then it follows that $\sum_{\gamma=\pm} \mathcal{J}^{\gamma}(\bbR^2) \lesssim G(T)^2,$ as desired. It remains to bound $\mathcal{K}^{\pm}(\bbR^2)$. It is immediate to see that
	\begin{equation*}
\mathcal{K}^{+}(\bbR^2) \lesssim \left\| \int_0^t e^{-\frac {t-\tau}2} \sqrt{|\xi_1|} \big| \wwhat{B \otimes \mathbf{u}}(\xi,\tau) \big\|_{L^1_{\xi_2}} \, \mathrm{d}\tau\right\|_{L_T^2 L^2_{\xi_1}} \lesssim \left\| \int_{\bbR^2}\sqrt{|\xi_1|} \big| \wwhat{B\otimes \mathbf{u}} \big| \, \mathrm{d}\xi \right\|_{L_T^2 L^2_{\xi_1}L^1_{\xi_2}}  \lesssim G(T)^2. 
\end{equation*} Exploiting \eqref{Omg_1} and \eqref{Omg_2}, we get
	\begin{equation*}
\mathcal{K}^{-}(\Omega_2\cup\Omega_3)	\lesssim \left\| \int_0^t e^{-\frac {t-\tau}2} \sqrt{|\xi_1|} \big| \wwhat{B \otimes \mathbf{u}}(\xi,\tau) \big\|_{L^1_{\xi_2}} \, \mathrm{d}\tau\right\|_{L^2(0,T;L^2_{\xi_1})}  \lesssim G(T)^2.
	\end{equation*}
	On the other hand, \eqref{Omg_3} and the definition of $\bfP$ gives
	\begin{equation*}
		\big| \langle \mathbf{P}(\xi)\wwhat{(B\cdot \nabla)\widetilde{\mathbf{u}}}, \mathbf{a}_{-}^2 \rangle \langle \mathbf{b}_{-}^2,e_2 \rangle \big|\leq C \big| \wwhat{\partial_1^2(B\otimes v)} \big| + C \big| \wwhat{\partial_1^3(B\otimes B)} \big|, \qquad \xi \in\Omega_3.
	\end{equation*}
 Such an observation allows us to have that $\mathcal{K}^{-}(\Omega_3)$ is bounded by a positive constant times
 \begin{equation*}
     \begin{split}
          \left\| \int_0^t \big| \xi_1^{\frac 32} e^{-\xi_1^2(t-\tau)} \wwhat{B\otimes v}(\xi,\tau) \big| \, \mathrm{d}\tau \right\|_{L_T^2 L^2_{\xi_1}L^1_{\xi_2}} + \left\| \int_0^t \big| \xi_1^{\frac 32} e^{-\xi_1^2(t-\tau)} \wwhat{\partial_1B\otimes B}(\xi,\tau) \big| \, \mathrm{d}\tau \right\|_{L_T^2 L^2_{\xi_1}L^1_{\xi_2}},
     \end{split}
 \end{equation*}
 which is bounded by a positive constant times $G(T)^2 + \| \widehat{\partial_1 B} \|_{L_T^{\frac{4}{3}} L^1} G(T)$. This gives the desired bound $\sum_{\gamma=\pm}\mathcal{K}^{\gamma}(\bbR^2) \lesssim \| \widehat{\partial_1 B} \|_{L_T^{\frac{4}{3}} L^1} G(T) + G(T)^2.$ Collecting all the previous estimates, we establish the desired bound for $\| \frac {|\widehat{v}_2|}{\sqrt{|\xi_1|}} \|_{L^2(0,T;L^2_{\xi_1}L^1_{\xi_2})}.$ Finally, we estimate $\| \widehat{v} \|_{L_T^{\frac 43} L^1}$. Applying Duhamel's principle to the $v$ equations in \eqref{mhd_eq}, we have $$\widehat{v} = e^{-t} \widehat{v}_0 - \int_0^t e^{-(t-\tau)} \wwhat{\bbP\, (v \cdot \nabla)v} \,\ud \tau + \int_0^t e^{-(t-\tau)} \wwhat{\bbP \, (B\cdot \nabla)B} \,\ud \tau + \int_0^t e^{-(t-\tau)} \wwhat{\partial_1B} \,\ud \tau.$$ With the help of Young's convolution inequality, one can easily show that $$\| e^{-t} \widehat{v}_0 \|_{L_T^{\frac 43} L^1} \leq \| v_0 \|_{H^m},$$ $$\left\| \int_0^t e^{-(t-\tau)} \wwhat{\bbP\, (v \cdot \nabla)v} \,\ud \tau \right\|_{L_T^{\frac 43} L^1} \lesssim \| \wwhat{(v \cdot \nabla)v} \|_{L_T^{\frac 43} L^1} \lesssim G(T)^2,$$ and $$\left\| \int_0^t e^{-(t-\tau)} \wwhat{\partial_1B} \,\ud \tau \right\|_{L_T^{\frac 43} L^1} \lesssim \| \wwhat{\partial_1B} \|_{L_T^{\frac 43} L^1}.$$ The algebraic relationship $(B \cdot \nabla)B = B_1 \partial_1B + B_2 \partial_2 B$ allows us to compute as 
	\begin{align*}
	    \left\| \int_0^t e^{-(t-\tau)} \wwhat{\bbP\, (B \cdot \nabla)B} \,\ud \tau \right\|_{L_T^{\frac 43} L^1} &\lesssim \| \wwhat{(B \cdot \nabla)B} \|_{L_T^{\frac 43} L^1} \\
	    &\lesssim \| \widehat{B_1} \|_{L^4_T L^1} \| \partial_1 B \|_{L^2_T L^1} +  \| \widehat{B}_2 \|_{L^2_T L^1} \| \widehat{\partial_2B} \|_{L^4_T L^1} \\
	    &\lesssim G(T)^2 + G(T)\| \widehat{B}_2 \|_{L^2_T L^1}.
	\end{align*}
This finishes the proof.

\subsection{Proof of Remark~\ref{X^m}}\label{sec_rmk}
We compute
\begin{equation*}
    \begin{split}
        \|\frac{\widehat{B_0}}{|\xi|}\|_{L^1(\bbR^2)} = \|\frac{\widehat{B_0}}{|\xi|}\|_{L^1(|\xi|< 1)} + \|\frac{\widehat{B_0}}{|\xi|}\|_{L^1(|\xi|\geq 1)} \lesssim \|\widehat{B_0}\|_{L^\infty} + \|\widehat{B_0}\|_{L^1(|\xi|\geq 1)} \lesssim \|B_0\|_{L^1} + \|B_0\|_{H^m}.
    \end{split}
\end{equation*}
In the above calculation, we use the integrability of $|\xi|^{-1}$ in the region $|\xi|<1$ and  the bound $|\xi|^{-1}\leq 1$ when $|\xi|\geq 1$. The integrability of $\frac{1}{(1+|\xi|^2)^{m}}$ for $|\xi|\geq 1$ is also used. Now we bound $\lnrm{|\xi|^{-1}\sqrt{|\xi_1|}\widehat{\mathbf{u}_0}}_{L_{\xi_1}^2 L_{\xi_2}^1}$ based on the decomposition $\bbR^2 = D_1 \cup D_2 \cup D_3$, where $D_1 = \{ |\xi_1| \geq 1 \}$, $D_2 = \{|\xi_1| \leq 1\} \cap \{|\xi_2| \leq 1\}$, and $D_3 = \bbR^2\setminus(D_1\cup D_2)$.
 It is immediate to see that
\begin{equation*}
    \begin{split}
        \lnrm{|\xi|^{-1}\sqrt{|\xi_1|}\widehat{\mathbf{u}_0}}_{L_{\xi_1}^2 L_{\xi_2}^1(\bbR^2)} &\lesssim \|\wwhat{\mathbf{u}_0}\|_{L_{\xi_1}^2 L_{\xi_2}^1(\bbR^2)} + \|\wwhat{\mathbf{u}_0}\|_{L^\infty}\||\xi_2|^{-\frac12}\|_{L_{\xi_2}(|\xi_2|\leq 1)} + \lnrm{\frac{\|\sqrt{|\xi_1|}\wwhat{\mathbf{u}_0}\|_{L_{\xi_1}^2}}{|\xi_2|}}_{L_{\xi_2}(|\xi_2|\geq 1)} \\
        &\lesssim \|\mathbf{u}_0\|_{H^m} + \|\mathbf{u}_0\|_{L^1} +\|\sqrt{|\xi_1}\widehat{\mathbf{u}_0}\|_{L^2} \\
        &\lesssim \|\mathbf{u}_0\|_{H^m\cap L^1},
    \end{split}
\end{equation*}
where we use, for the first inequality, Minkowski's integral inequality and the three simple facts $\frac{\sqrt{|\xi_1|}}{|\xi|}\lesssim 1$ for $\xi\in D_1$, $\frac{\sqrt{|\xi_1|}}{|\xi|} \lesssim \frac{1}{\sqrt{|\xi_2|}}$ for $\xi \in D_2$, and $\frac{\sqrt{|\xi_1|}}{|\xi|}\lesssim \frac{\sqrt{|\xi_1|}}{|\xi_2|}$ for $\xi\in D_3$. In the second inequality, we just observe $|\xi_2|^{-1} \lesssim 1$ if $|\xi_2|\geq 1$.  
It remains to control $\||\xi_1|^{-\frac12}\widehat{B_0}\|_{L^1}.$ 
We can similarly show as before that 
\begin{equation*}
    \int_{D_1} \frac{1}{\sqrt{|\xi_1|}} |\mathscr{F} B_0| \,\ud \xi \leq \int_{\bbR^2} |\mathscr{F} B_0| \,\ud \xi \leq \| B_0 \|_{H^m}
\end{equation*}
and
\begin{equation*}
    \int_{D_2} \frac{1}{\sqrt{|\xi_1|}} |\mathscr{F} B_0| \,\ud \xi \leq \| \mathscr{F} B_0 \|_{L^{\infty}} \int_{D_2} \frac{1}{\sqrt{|\xi_1|}} \,\ud \xi \leq C \| B_0 \|_{L^1}.
\end{equation*}
To estimate the integral on $D_3$, we need to divide it into the two parts $\{|\xi_1||\xi_2|^4 \geq 1\}$ and $\{|\xi_1||\xi_2|^4 \leq 1\}$ again. In the first case, we have for $m > 3$ that $$\int_{D_3 \cap \{|\xi_1||\xi_2|^4 \geq 1\}} \frac{1}{\sqrt{|\xi_1|}} |\mathscr{F} B_0| \,\ud \xi \leq \int_{\bbR^2} |\xi_2|^2 |\mathscr{F} B_0| \,\ud \xi \leq \| \mathscr{F}B_0 \|_{H^m}.$$ Meanwhile, for any $\epsilon \in (0,\frac{1}{4})$,  it holds $$\int_{D_3 \cap \{|\xi_1||\xi_2|^4 \leq 1\}} \frac{1}{\sqrt{|\xi_1|}} |\mathscr{F} B_0| \,\ud \xi \leq \| \mathscr{F}B_0 \|_{L^{\infty}} \int_{D_3} |\xi_1|^{-(1-\epsilon)} |\xi_2|^{-2(1-2\epsilon)} \,\ud \xi \leq C \| B_0 \|_{L^{1}}.$$ Thus, we obtain the claim.

\bibliographystyle{amsplain}
\bibliography{2DMHD}

\end{document}